\documentclass[11pt]{article}
\usepackage[margin=1in]{geometry} 
\usepackage{amssymb,amsfonts,amsmath,bbm,mathrsfs,stmaryrd}
\usepackage{xcolor}
\usepackage{url}

\usepackage[colorlinks,
             linkcolor=black!75!red,
             citecolor=blue,
             pdftitle={Quantum rigidity of negatively curved manifolds},
             pdfauthor={Alexandru Chirvasitu},
             pdfproducer={pdfLaTeX},
             pdfpagemode=None,
             bookmarksopen=true
             bookmarksnumbered=true]{hyperref}

\usepackage{tikz}
\usetikzlibrary{arrows,calc,decorations.pathreplacing,decorations.markings,intersections,shapes.geometric,through,fit,shapes.symbols,positioning,decorations.pathmorphing}

\usepackage[amsmath,thmmarks,hyperref]{ntheorem}
\usepackage{cleveref}

\crefname{section}{Section}{Sections}
\crefformat{section}{#2Section~#1#3} 
\Crefformat{section}{#2Section~#1#3} 

\crefname{subsection}{\S}{\S\S}
\crefformat{subsection}{#2\S#1#3} 
\Crefformat{subsection}{#2\S#1#3} 

\theoremstyle{plain}

\newtheorem{lemma}{Lemma}[section]
\newtheorem{proposition}[lemma]{Proposition}
\newtheorem{corollary}[lemma]{Corollary}
\newtheorem{theorem}[lemma]{Theorem}

\theoremstyle{nonumberplain}
\newtheorem{theoremN}{Theorem}

\theoremstyle{plain}
\theorembodyfont{\upshape}
\theoremsymbol{\ensuremath{\blacklozenge}}

\newtheorem{definition}[lemma]{Definition}

\newtheorem{remark}[lemma]{Remark}

\crefname{definition}{definition}{definitions}
\crefformat{definition}{#2definition~#1#3} 
\Crefformat{definition}{#2Definition~#1#3} 

\crefname{ex}{example}{examples}
\crefformat{example}{#2example~#1#3} 
\Crefformat{example}{#2Example~#1#3} 

\crefname{remark}{remark}{remarks}
\crefformat{remark}{#2remark~#1#3} 
\Crefformat{remark}{#2Remark~#1#3} 

\crefname{convention}{convention}{conventions}
\crefformat{convention}{#2convention~#1#3} 
\Crefformat{convention}{#2Convention~#1#3}

\crefname{lemma}{lemma}{lemmas}
\crefformat{lemma}{#2lemma~#1#3} 
\Crefformat{lemma}{#2Lemma~#1#3} 

\crefname{proposition}{proposition}{propositions}
\crefformat{proposition}{#2proposition~#1#3} 
\Crefformat{proposition}{#2Proposition~#1#3} 

\crefname{corollary}{corollary}{corollaries}
\crefformat{corollary}{#2corollary~#1#3} 
\Crefformat{corollary}{#2Corollary~#1#3} 

\crefname{theorem}{theorem}{theorems}
\crefformat{theorem}{#2theorem~#1#3} 
\Crefformat{theorem}{#2Theorem~#1#3} 

\crefname{assumption}{assumption}{Assumptions}
\crefformat{assumption}{#2assumption~#1#3} 
\Crefformat{assumption}{#2Assumption~#1#3} 

\crefname{equation}{}{}
\crefformat{equation}{(#2#1#3)} 
\Crefformat{equation}{(#2#1#3)}

\theoremstyle{nonumberplain}
\theoremsymbol{\ensuremath{\blacksquare}}

\newtheorem{proof}{Proof}

\newtheorem{proof_of_fin_orb}{Proof of \Cref{th.fin_orb=>classical}}
\newtheorem{proof_of_main}{Proof of \Cref{th.main}}

\newtheorem{proof of sbgp}{Proof of \Cref{th.sbgp}}

\newcommand\bC{{\mathbb C}}

\newcommand\bR{{\mathbb R}}

\newcommand\cC{{\mathcal C}}

\newcommand\cM{{\mathcal M}}
\newcommand\cT{{\mathcal T}}
\newcommand\cW{{\mathcal W}}

\newcommand\ol{\overline}
\newcommand\wt{\widetilde}

\DeclareMathOperator{\id}{id}

\DeclareMathOperator{\pr}{\mathrm{Prob}}
\DeclareMathOperator{\orb}{\mathrm{Orb}}

\newcommand{\define}[1]{{\em #1}}

\newcommand{\cat}[1]{\textsc{#1}}

\newcommand{\qedhere}{\mbox{}\hfill\ensuremath{\blacksquare}}

\title{Quantum rigidity of negatively curved manifolds}
\author{Alexandru Chirvasitu\footnote{University of Washington, \url{chirva@uw.edu}}}

\begin{document}

\date{}

\maketitle

\begin{abstract}
We show that an isometric action of a compact quantum group on the underlying geodesic metric space of a compact connected Riemannian manifold $(M,g)$ with strictly negative curvature is automatically classical, in the sense that it factors through the action of the isometry group of $(M,g)$. This partially answers a question by D. Goswami. 
\end{abstract}

\noindent {\em Key words: compact quantum group, compact metric space, isometric coaction, Riemannian manifold, geodesic distance, negatively curved}

\tableofcontents

\section*{Introduction}

This paper is essentially a continuation of \cite{Chi15}, and it similarly deals with quantum symmetries of compact metric spaces. 

The framework is that of compact quantum groups, as introduced by Woronowicz in \cite{Wor87}. These are Hopf algebras with some additional structure and properties tailor-made to imitate algebras of sufficiently well-behaved functions on compact groups. They fit well within the wider scheme of non-commutative geometry \cite{Con94}, and should be thought of as algebras of functions on ``non-commutative spaces'' equipped with a group structure.  

Much of the theory of compact groups generalizes to the quantum setting, and the field has developed quite explosively. One prevalent point of view (and the one we adopt here) is that compact quantum groups should be thought of as implementing ``quantum symmetries'' of various algebraic or geometric structures. Examples of such structures for which this works well and produces interesting examples abound: finite-dimensional Hilbert spaces \cite{DaeWan96}, finite graphs \cite{Bic03}, operator algebras \cite{Boc95,Wan98,Wan99}, Riemannian manifolds \cite{BhoGos09} and so on. In fact, even as simple a structure as a finite set has interesting quantum automorphisms. This is the conclusion of \cite{Wan98}, where the quantum automorphism group of a finite space is shown to be strictly larger than the usual permutation group. 

One ubiquitous type of structure in geometry is that of a metric space. Banica first introduced in \cite{Ban05} the notion of isometric action of a quantum group on a finite metric space, and the concept was later extended to arbitrary compact metric spaces in \cite{Gos12}. In rough terms (to be made precise in \Cref{subse.prel_metric}) an action of a quantum group on a compact metric space $(X,d)$ is \define{isometric} if the distance function $d$, regarded as a function on $X\times X$, is invariant under a kind of diagonal action.

The concept of a distance function, which is global in nature, has an infinitesimal analogue: that of a Riemannian metric. Structure-preserving actions of quantum groups on Riemannian manifolds were defined earlier in \cite{BhoGos09}, and one striking phenomenon discovered in \cite{DasGos13} is that compact, connected Riemannian manifolds have no truly quantum symmetries: Any structure-preserving action by a quantum group on such a manifold factors through an isometric action of an ordinary compact group (we also say that the action is \define{classical}). 

A Riemannian manifold $X$ can be naturally made into a metric space by means of the geodesic distance. One is now confronted with two possible notions of quantum isometry: the global one of \cite{Gos12} and the infinitesimal one of \cite{BhoGos09}. They are only conjecturally equivalent, with the global notion being weaker. Taking a cue from \cite{DasGos13}, one can then ask the same type of rigidity question: Are all isometric quantum actions on the underlying metric space of a compact connected Riemannian manifold classical? 

One of the main results of the paper is that indeed they are under a curvature constraint (see \Cref{th.main}):

\begin{theoremN}
The underlying metric space of a compact, connected Riemannian manifold with negative sectional curvature  has no truly quantum symmetries.  
\end{theoremN}

The curvature condition is crucial in the proof, but along the way we develop some geometric techniques that one might hope are of some interest in their own right and could be useful in studying quantum isometries of more general manifolds. 

One final observation is that the type of rigidity problem studied here, of when a quantum action is classical, has counterparts in somewhat different settings that suggest some general principles are at work. In other words, the `quantum rigidity' in the title (i.e. the absence of truly quantum symmetries) is a phenomenon of some scope and interest.  

To give just one example, a coaction of a sufficiently nice finite-dimensional Hopf algebra on a commutative domain is classical (this is the main result of \cite{EW14}). There do not seem to be direct implications between this and the previously-mentioned results, but the analogy is clear: \cite{EW14} is an algebro-geometric version of the rigidity theorems mentioned above, which are grounded in metric / Riemannian geometry.

The paper is organized as follows: 

\Cref{se.prel} contains preparatory remarks and results on compact quantum groups and their actions on (mostly classical) compact spaces. 

In \Cref{se.orbits} we study the notion of orbit for an action of a quantum group on a compact space. The concept is ambiguous, as the various definitions proposed in \cite{Hua12} are not known to coincide. Nevertheless, we prove that they do when the action in question is isometric. 

The main result of \Cref{se.criterion} (\Cref{pr.criterion}) introduces a sufficient condition for an isometric action of a quantum group on a metric space to factor through an action of an ordinary compact group (i.e. for the action to be classical). Intuitively, the condition demands that there be a point whose behavior under the action ``determines the action'' entirely. 

The typical situation to keep in mind is that of a faithful action $X\times G\to X$ of a finite group $G$ on a compact connected Hausdorff space $X$: There must be a point $x\in X$ with trivial isotropy group in $G$. This will then imply that elements of $G$ can be distinguished by what they do to $x$ alone.

The first main result of the paper is \Cref{th.fin_orb=>classical} in \Cref{se.fin_orb=>classical}. It states that if the metric space acted upon is the underlying geodesic metric space of a compact, connected Riemannian manifold and the action has finite orbits, then it is classical.

Finally, in \Cref{se.neg_curv} we show that if the Riemannian manifold acted upon isometrically by a quantum group has negative sectional curvature everywhere, then the finiteness condition in \Cref{th.fin_orb=>classical} is automatically satisfied. In conclusion, such actions are always classical.

\subsection*{Acknowledgements}

I am grateful to Debashish Goswami for many valuable discussions on the contents of \cite{Gos12}.

\section{Preliminaries}\label{se.prel}

We will assume some basic working knowledge of operator algebras ($C^*$ and von Neumann), as amply covered by any number of texts on the topic (say \cite{Tak02}). Unless specified otherwise, all algebras are unital and associative. 

Some notational conventions in place throughout the paper:

\begin{itemize}
  \item The tensor product symbol `$\otimes$' signifies the ordinary, algebraic tensor product when placed between algebras with no topological structure, and the minimal or injective $C^*$ tensor product when appearing between two $C^*$-algebras (\cite[Chapter IV, Definition 4.8]{Tak02}).
  \item For an algebra $A$, the symbol $A^*$ similarly changes meaning depending on the nature of $A$: If it is simply an algebra with no extra topological structure, then $A^*$ is the full dual, consisting of all linear functionals on $A$. If on the other hand $A$ is a $C^*$-algebra, then $A^*$ is the set of continuous linear functionals. 
  \item The set of \define{states} on a $C^*$-algebra $A$ (that is, unital, positive continuous functionals) is denoted by $S(A)$.
  \item Measures always live on metrizable compact Hausdorff spaces, so in particular they are regular. 
  \item For a space $X$ we will typically conflate points $x\in X$ and the corresponding Dirac delta measures $\delta_x$, and hence tacitly regard $X$ as a subspace of its space $\pr(X)$ of probability measures. 
\end{itemize}

\subsection{Quantum groups and actions}\label{subse.prel_cqg}

Recal the notion of compact quantum group in wide use throughout the literature (e.g. \cite{Wan98,QuaSab12,Gos12} or the survey \cite{KusTus99}). We will generally omit the word `compact' and simply refer to these objects as `quantum groups'.

\begin{definition}\label{def.cqg}
A \define{quantum group} $(A,\Delta)$ is a unital $C^*$-algebra endowed with a $C^*$-algebra homomorphism $\Delta:A\to A\otimes A$ satisfying the following properties:
\begin{enumerate}
\renewcommand{\labelenumi}{(\arabic{enumi})}
 \item $\Delta$ is coassociative, in the sense that the diagram
   \[
     \tikz[anchor=base]{
       \path (0,0) node (1) {$A$} +(2.5,.5) node (2) {$A\otimes A$} +(2.5,-.5) node (3) {$A\otimes A$} +(5.5,0) node (4) {$A\otimes A\otimes A$};
       \draw[->] (1) to[bend left=10] node[pos=.5,auto] {$\scriptstyle \Delta$} (2);
       \draw[->] (1) to[bend right=10] node[pos=.5,auto,swap] {$\scriptstyle \Delta$} (3);
       \draw[->] (2) to[bend left=10] node[pos=.3,auto] {$\scriptstyle \Delta\otimes\id$} (4);
       \draw[->] (3) to[bend right=10] node[pos=.3,auto,swap] {$\scriptstyle \id\otimes\Delta$} (4);
     }
   \]commutes;
 \item The subspaces
   \[
     \text{linear span}\{(a\otimes 1)\Delta(b)\ |\ a,b\in A\}
   \]
   and
   \[
     \text{linear span}\{(1\otimes a)\Delta(b)\ |\ a,b\in A\}
   \]
   are norm-dense in $A\otimes A$. 
\end{enumerate} 
\end{definition}

We will often abuse terminology by just writing $A$ for $(A,\Delta)$. 

Two functionals $\varphi,\psi\in A^*$ can be multiplied by means of the comultiplication: By definition, $\varphi\psi$ is the composition
\begin{center}
\begin{tikzpicture}[auto]
  \node (1) at (0,0) {$A$};
  \node (2) at (2,.5) {$A\otimes A$};
  \node (3) at (4,0) {$\bC$};
  \draw[->,bend left=15] (1) to node {$\scriptstyle \Delta$} (2);
  \draw[->,bend left=15] (2) to node {$\scriptstyle \varphi\otimes\psi$} (3);
  \draw[->,bend right=20] (1) to node [swap] {$\scriptstyle \varphi\psi$} (3);
\end{tikzpicture}
\end{center}  
This makes $A^*$ into a semigroup (that is, the multiplication is associative), and $S(A)$ is a sub-semigroup.

The so-called \define{Haar state} of a quantum group $(A,\Delta)$ plays the same role that the Haar measure does for classical compact groups: We have  
\begin{equation}\label{eq.haar}
 h\varphi=\varphi h=h,\ \forall \varphi\in A^*.
\end{equation} 
The existence of the Haar state is proven in \cite[4.1]{Wor87} in the separable case and in \cite[2.4]{Dae95} in general.

\begin{remark}\label{rem.red0}
If the Haar state $h$ is faithful, the quantum group is \define{reduced}. 

Every quantum group $A$ surjects onto a reduced one $A_r$: Simply take $A_r$ to be the $C^*$-algebra obtained by the GNS construction applied to $A$ and $h$.   
\end{remark}

\subsection{Actions on spaces}\label{subse.prel_sp}

Recall from the introduction that we think of quantum groups as implementing ``quantum symmetries'' of various mathematical objects. Here, these are operator algebras thought of as function algebras of possibly non-commutative spaces, perhaps equipped with some additional structure. For this reason, the notion of action to be used in the sequel is as follows:

\begin{definition}\label{def.cqg_act}
Let $(A,\Delta)$ be a quantum group, and $B$ a $C^*$-algebra. A \define{right coaction} of $A$ on $B$ is a $C^*$ homomorphism $\rho:B\to B\otimes A$ such that
\begin{enumerate}
\renewcommand{\labelenumi}{(\arabic{enumi})}
 \item The diagram
   \[
    \tikz[anchor=base]{
     \path (0,0) node (1) {$B$} +(2.5,.5) node (2) {$B\otimes A$} +(2.5,-.5) node (3) {$B\otimes A$} +(5.5,0) node (4) {$B\otimes A\otimes A$};
       \draw[->] (1) to[bend left=10] node[pos=.5,auto] {$\scriptstyle \rho$} (2);
       \draw[->] (1) to[bend right=10] node[pos=.5,auto,swap] {$\scriptstyle \rho$} (3);
       \draw[->] (2) to[bend left=10] node[pos=.3,auto] {$\scriptstyle \rho\otimes\id$} (4);
       \draw[->] (3) to[bend right=10] node[pos=.3,auto,swap] {$\scriptstyle \id\otimes\Delta$} (4);
    }   
   \]commutes;
 \item The subspace
   \[
     \text{linear span}\{(1\otimes a)\rho(b)\ |\ a\in A\ b\in B\}
   \]   
   is dense in $B\otimes A$.
\end{enumerate} 
\end{definition}

\begin{remark}\label{rem.red}
The quantum group $A$ coacting on $B$ can always be assumed to be reduced if convenient: If $\tau:A\to A_r$ is the surjection from \Cref{rem.red0}, then the composition 
    \begin{equation*}\label{eq.red}
      \begin{tikzpicture}[auto,baseline=(current  bounding  box.center)]
        \node (1) at (0,0) {$B$};
        \node (2) at (2,0) {$B\otimes A$};
        \node (3) at (4.5,0) {$B\otimes A_r$};
        \draw[->] (1) to node {$\scriptstyle \rho$} (2);
        \draw[->] (2) to node {$\scriptstyle \id\otimes \tau$} (3);
        \draw[->] (1) to[bend right=20] node[pos=.5,swap] {$\scriptstyle \rho_r$} (3);
      \end{tikzpicture}
    \end{equation*}
can be shown to be a coaction again. 
\end{remark}

Below, $B$ will mostly be commutative; in other words, the algebra of continuous functions on a compact Hausdorff space.

%
%
%

\begin{definition}\label{def.cqg_act_bis}
Let $(A,\Delta)$ be a $C^*$ quantum group and $X$ a compact Hausdorff topological space. A (right) action of $A$ on $X$ is a coaction of $A$ on $\cC(X)$ as in \Cref{def.cqg_act}. 
\end{definition}

\begin{remark}
\Cref{def.cqg_act_bis} is justified by the fact that when $A$ is the function algebra of an ordinary compact group $G$, \Cref{def.cqg_act} does indeed specialize to a right action of $G$ on $X$.
\end{remark}

Given an action $\rho$ be an action of $A$ on $X$ as in \Cref{def.cqg_act_bis}, let $\varphi\in A^*$. The composition
\begin{center}
\begin{tikzpicture}[auto]
  \node (1) at (0,0) {$\cC(X)$};
  \node (2) at (2,0) {$\cC(X)\otimes A$};
  \node (3) at (4.5,0) {$\cC(X)$};
  \draw[->] (1) to node {$\scriptstyle \rho$} (2);
  \draw[->] (2) to node {$\scriptstyle \id\otimes\varphi$} (3);
\end{tikzpicture}
\end{center}  
 is a self-map on $\cC(X)$, which we denote by $\varphi\triangleright$. As $\varphi$ ranges over $A^*$, these maps constitute a left action of the semiring $A^*$ on the vector space $\cC(X)$. The restriction of this action to the sub-semigroup $S(A)\subset A^*$ will again be denoted by $\triangleright$.  
 
The semiring $A^*$ similarly acts on the right on the Banach space $\cM(X)=\cC(X)^*$ of complex measures on $X$. The result $\mu\triangleleft\varphi$ of acting with $\varphi\in A^*$ on $\mu\in\cM(X)$ is the measure defined by   
\begin{center}
\begin{tikzpicture}[auto]
  \node (1) at (0,0) {$\cC(X)$};
  \node (2) at (2,.5) {$\cC(X)\otimes A$};
  \node (3) at (4,0) {$\bC$};
  \draw[->,bend left=10] (1) to node {$\scriptstyle \rho$} (2);
  \draw[->,bend left=10] (2) to node {$\scriptstyle \mu\otimes\varphi$} (3);
  \draw[->,bend right=20] (1) to node [swap] {$\scriptstyle\mu\triangleleft\varphi$} (3);
\end{tikzpicture}
\end{center}  
$\triangleleft$ restricts to an action (denoted by the same symbol) of $S(A)$ on $\pr(X)=S(\cC(X))$.

%
%
%
%

We also need to know what it means for a quantum group to act faithfully.

\begin{definition}\label{def.faithful}
A coaction $\rho:B\to B\otimes A$ of a quantum group $A$ on a $C^*$-algebra $B$ is \define{faithful} if the set
\begin{equation}\label{eq.faithful}
 \{(\mu\otimes\id)(\rho(b))\ |\ b\in B,\ \mu\in S(B)\}\subset A
\end{equation}
generates $A$ as a $C^*$-algebra.  
\end{definition}

\begin{remark}\label{rem.faithful}
It is mostly harmless to assume a coaction is faithful: The $C^*$-subalgebra of $A$ generated by \Cref{eq.faithful} is itself a compact quantum group coacting on $B$, and $\rho$ factors through this latter coaction (which is faithful). 
\end{remark}

With this in place, we can remind the reader that reduced quantum groups acting faithfully on a classical compact space are \define{of Kac type}, in the sense that they admit an antipode. In other words, there is a bounded, involutive, multiplication-reversing map $\kappa:A\to A$ that plays the same role as the map $f\mapsto f(\bullet^{-1})$ for continuous functions $f$ on an ordinary compact group. This is \cite[Theorem 3.23]{Hua12}, and we refer to that paper for details on this point.

\subsection{Isometric actions}\label{subse.prel_metric}

We now specialize the discussion to compact metric spaces $(X,d)$, reserving the symbol $\rho$ for an action by a quantum group $(A,\Delta)$ on $X$. We always assume that $\rho$ is faithful and hence $A$ admits an antipode $\kappa$. 

This section gathers some auxiliary results on distance-preserving actions that are needed below. First, recall the following notion (cf. \cite[Definition 3.3]{Gos12}; such actions are called (D)-isometric in \cite[Definition 1.6]{Chi15}):

\begin{definition}\label{def.D_isometric}
For $x\in X$, let $d_x\in\cC(X)$ the function defined by $d_x(y) = d(x,y)$. If $A$ is reduced, the action $\rho$ is \define{isometric} if 
\begin{equation}\label{eq.D_isometric}
 \rho(d_y)(x) = \kappa(\rho(d_x)(y)),\ \forall x,y\in X, 
\end{equation}
where for $f\in \cC(X)$, $\rho(f)(x)$ is the evaluation at $x$ on the left hand tensorand of the element $\rho(f)\in\cC(X)\otimes A$. 

In general, $\rho$ is isometric if its reduced version $\rho_r$ defined in \Cref{rem.red} is isometric in the sense of the above paragraph. 
\end{definition}

\begin{remark}
This is the quantum analogue of the usual requirement that 
\begin{equation*}\label{eq.classical_isometry}
 d(xg,y) = d(x,yg^{-1}),\ \forall x,y\in X,\ \forall g\in G. 
\end{equation*}
for a right action of a compact group $G$ on $X$. 
\end{remark}

In order to state the next theorem we need some terminology.

\begin{definition}
For probability measures $\mu$ and $\nu$ on $X$ the set $\Pi(\mu,\nu)$ of \define{$(\mu,\nu)$-couplings} is the set of measures $\pi\in \pr(X\times X)$ such that 
\[
 p_{1*}(\pi)=\mu,\quad p_{2*}(\pi)=\nu,
\] 
where $p_i$ are the projections $X\times X\to X$. 
\end{definition}

\cite[Theorem 3.1]{Chi15} (slightly paraphrased) now reads:

\begin{theorem}\label{th.old}
If an action $\rho$ of a quantum group $A$ on a compact metric space $(X,d)$ is isomeric, then for every two points $x,y\in X$ and every state $\psi\in S(A_r)$ there is a coupling $\pi\in\Pi(x\triangleleft\psi,y\triangleleft\psi)$ supported on 
\[
 \{(x',y')\in X\times X\ |\ d(x',y')=d(x,y)\}.
\]
\qedhere
\end{theorem}

\begin{remark}
Here, states of $A_r$ act on $\pr(X)$ via the reduced action $\rho_r$ from \Cref{rem.red}. Equivalently, it is the action of $S(A)$ on $\pr(X)$ (induced by $\rho$) restricted to $S(A_r)\subseteq S(A)$.   
\end{remark}

We make one last observation on what the conclusion of \Cref{th.old} entails. 

Recall first that the \define{Kantorovi\u c distance} on $\pr(X)$ \cite{Kan42,KanRub57} is an extension of $d$ defined by  
\begin{equation}\label{eq.KanRub}
 d(\mu,\nu) = \sup_{L(f)\le 1}(\mu(f)-\nu(f)),\ \forall\mu,\nu\in\pr(X),
\end{equation}
where $\mu(f)=\int_Xf\ \text{d}\mu$ and $L(f)$ is the Lipschitz constant of $f$:
\[
 L(f) = \sup_{x\ne y}\frac{|f(x)-f(y)|}{d(x,y)}.
\]

It is a result due to Kantorovi\u c and Rubinstein \cite{KanRub57} that the distance $d$ defined by \Cref{eq.KanRub} has another expression as
\begin{equation*}
d(\mu,\nu) = \inf_{\pi\in\Pi(\mu,\nu)}\int_{X\times X}d(x,y)\ \mathrm{d}\pi. 
\end{equation*}
As a consequence of this and \Cref{th.old} we have

\begin{corollary}\label{cor.old}
Under the assumptions of \Cref{th.old} we have
\[
 d(\mu\triangleleft\psi,\nu\triangleleft\psi)\le d(\mu,\nu) 
\] 
for all $\mu,\nu\in\pr(X)$ and $\psi\in S(A_r)$.
\qedhere
\end{corollary}

\section{Orbits}\label{se.orbits}

We gather together the various notions of orbit for an action of a quantum group on a compact space and show in this section that they all coincide for isometric actions. 

Let us first review the three competing notions of orbit introduced in \cite[$\S$4.2]{Hua12}. The notations from above are still in place, with $A$ being a compact quantum group acting faithfully on $X$ via $\rho:\cC(X)\to \cC(X)\otimes A$ and $h\in A^*$ the Haar state on $A$. Recall also that we are denoting the antipode of $A$ by $\kappa$. 

Classically, the orbit of a point $x\in X$ under the action of a compact group is the result of ``smearing'' $x$ around via the action. It turns out there are several ways to make sense of this in the half-quantum situation when $A$ is no longer (the function algebra of) an ordinary compact group, and in general they are only conjecturally equivalent:

\begin{enumerate}
\renewcommand{\labelenumi}{(\alph{enumi})}
  \item One possibility would be to declare the orbit to be the support of the measure $x\triangleleft h\in\pr(X)$; this measure is denoted by $\mu_x$ in \cite[$\S$4.2]{Hua12}
  \item Another option: Consider the composition
    \begin{equation}\label{eq.comp}
      \begin{tikzpicture}[auto,baseline=(current  bounding  box.center)]
        \node (1) at (0,0) {$\cC(X)$};
        \node (2) at (2,0) {$\cC(X)\otimes A$};
        \node (3) at (4.5,0) {$A$};
        \draw[->] (1) to node {$\scriptstyle \rho$} (2);
        \draw[->] (2) to node {$\scriptstyle x\otimes\id$} (3);
      \end{tikzpicture}
    \end{equation}  
    where $x$ here is though of as the Dirac measure $\delta_x:\cC(X)\to \bC$. The image of the composition is a quotient $C^*$-algebra of $\cC(X)$, so it is the function algebra of a subspace of $X$. Declare this subspace to be the orbit; this is $\cM_x$ from \cite[Definition 4.6]{Hua12}.
  \item Finally, we can define an equivalence relation $\simeq$ on $X$ by 
\[
 x\simeq x'\iff x\triangleleft h=x'\triangleleft h
\]
and define the orbits to be the equivalence classes; the class of $x$ is denoted by $\orb_x$ in loc. cit. 
\end{enumerate}

\cite[Conjecture 4.19]{Hua12} says that all of these coincide for an arbitrary action. We prove this conclusion below under the assumption (in place from now on) that the topology of $X$ is given by a metric $d$ with respect to which $\rho$ is isometric. We can thus refer unambiguously to the orbit of a point under $\rho$.

Denote the closed ball of radius $r\ge 0$ centered at $x\in X$ by $B(x,r)$ and $\varphi\circ\kappa$ by $\ol{\varphi}$ for $\varphi\in A^*$.

\begin{proposition}\label{pr.orbits}
The action $\rho$ is isometric if and only if 
\begin{equation}\label{eq.pr.orbits}
 (x\triangleleft\varphi)(B(y,r)) = \left(y\triangleleft\ol{\varphi}\right)(B(x,r))
\end{equation}
for all $x,y\in X$, $\varphi\in S(A_r)$ and non-negative $r$.  
\end{proposition}
\begin{proof}
The result is essentially contained in \cite{Chi15}, but we need to unpack the language used there. Let us assume throughout the proof that $A$ is reduced (see \Cref{rem.red}). 

Because it is sometimes more convenient to work with Borel functions (such as characteristic functions of balls in a metric space), lift the algebra morphism \Cref{eq.comp} to a map $\cW(X)\to \ol{A}$, where $\cW(X)$ is the enveloping von Neumann algebra of $\cC(X)$ (\cite[Section III.2]{Tak02}) and $\ol{A}$ is the enveloping von Neumann algebra of $A$. For every Borel set $S\subseteq X$, denote the image of the characteristic function $\chi_S$ through this map by $a_{x;S}$. 

With this in place, \cite[Lemma 3.4]{Chi15} says that the action $\rho$ is isometric (the term there is `(D)-isometric') if and only if
\[
 a_{x;B(y,r)} = \kappa\left(a_{y;B(x,r)}\right),\ \forall x,y\in X,\ \forall r\ge 0.
\] 
Now simply evaluate $\varphi\in S(A)$ against this equality; the result is exactly \Cref{eq.pr.orbits}.
\end{proof}

\begin{corollary}\label{cor.orbits}
If $\rho$ is isometric and $\varphi\in S(A_r)$ and $y=x\triangleleft\varphi$ is a point, then $y\triangleleft\ol{\varphi}=x$. 
\end{corollary}
\begin{proof}
Applying \Cref{pr.orbits} to the points $x$ and $y$ with $r=0$ will result in $\left(y\triangleleft\ol{\varphi}\right)(x)=1$, meaning that the probability measure $y\triangleleft\ol{\varphi}$ is supported at $x$.  
\end{proof}

\begin{corollary}\label{cor.orbits2}
For isometric actions the definitions (a), (b) and (c) from above coincide. Moreover, all are equal to the set 
\begin{equation*}\label{eq.cor.orbits2}
 \cM'_x=\{y\in X\ |\ \exists \varphi\in S(A_r) \text{ such that } x\triangleleft \varphi=y\}.
\end{equation*}
as well as 
\begin{equation*}\label{eq.cor.orbits2_bis}
 \cM''_x=\{y\in X\ |\ \exists \varphi\in S(A) \text{ such that } x\triangleleft \varphi=y\}.
\end{equation*}
\end{corollary}
\begin{proof}
According to \cite[Theorem 4.18]{Hua12} we have $\mathrm{supp}~\mu_x\subseteq \cM_x\subseteq \orb_x$ and moreover $\mathrm{supp}~\mu_x=\cM_x$ if $A$ is reduced. 

{\bf Step 1: The sets $\cM'_x$ form a partition of $X$.} Define a relation $\sim$ on $X$ by
\[
 z\sim z' \iff \exists \varphi\in S(A_r) \text{ such that } z\triangleleft \varphi =z'.
\] 
It is transitive because $\triangleleft$ is an action of the semigroup $S(A_r)$, and it is symmetric and reflexive by \Cref{pr.orbits}. Symmetry is immediate; to verify reflexivity, note first that there is always a state $\varphi$ such that $x\triangleleft \varphi$ is a point. Indeed, simply select a $\varphi$ that restricts to a pure state on the image of \Cref{eq.comp}. But now \Cref{cor.orbits} says that $x\triangleleft\left(\varphi\ol{\varphi}\right)=x$, hence the reflexivity claim. 

We now know that $\sim$ is an equivalence relation, and the sets of the form \Cref{eq.cor.orbits2} are its classes. This concludes the proof of Step 1. 

{\bf Step 2: $\cM_x = \cM''_x$.} This is almost tautological: $\cM_x$ is by definition the spectrum of the image of \Cref{eq.comp}, i.e. the set of pure states of this image. These are exactly the multiplicative states $y\in X$  on $\cC(X)$ that factor through \Cref{eq.comp}, i.e. of the form $\varphi\circ \Cref{eq.comp}$ for some state $\varphi$ on $A$. Finally, $y=\varphi\circ\Cref{eq.comp}$ precisely means $y=x\triangleleft\varphi$. This finishes Step 2.  

{\bf Step 3: The reduced case.} Now consider the reduced version $\rho_r$ of the action (\Cref{rem.red}). The sets $\cM'_x$ and $\cM''_x$ for it coincide, so Steps 1 and 2 ensure that the sets $\cM_x$ (for $\rho_r$) partition $X$. But then, since $\orb_x$ also partition $X$ by definition, each $\orb_x$ comprises several $\cM_y$. Since the latter however contains the support of $\mu_y=\mu_x$, $\orb_x$ cannot possibly consist of several disjoint $\cM_y$ and so $\mathrm{supp}~\mu_x=\cM_x=\orb_x$ when $A$ is reduced.

{\bf Step 4: The general case.} Since $\mathrm{supp}~\mu_x$ and $\orb_x$ do not change when passing to the reduced version of the action and $\cM_x$ is always trapped between them, Step 3 shows that $\mathrm{supp}~\mu_x=\cM_x=\orb_x$ for arbitrary isometric actions. 

This set is also equal to $\cM''_x$  by Step 2, and similarly it is equal to $\cM'_x$ because passage to the reduced action $\rho_r$ identifies $\cM'_x$ with $\cM''_x$ and does not change $\orb_x$.
\end{proof}

\section{A criterion for an action to be classical}\label{se.criterion}

This is a short section and a bit of a detour, in that it is not concerned with metric structures. The title means the following:

\begin{definition}\label{def.classical}
An action $\rho$ of a quantum group $A$ on a compact space $X$ is \define{classical} if it can be realized as a composition
\begin{equation*}\label{eq.classical}
      \begin{tikzpicture}[auto,baseline=(current  bounding  box.center)]
        \node (1) at (0,0) {$\cC(X)$};
        \node (2) at (3,.5) {$\cC(X)\otimes \cC(G)$};
        \node (3) at (6,0) {$\cC(X)\otimes A$};
        \draw[->] (1) to[bend left=10] node[pos=.5,auto] {$\scriptstyle \rho'$} (2);
        \draw[->] (2) to[bend left=10] node[pos=.5,auto] {$\scriptstyle \id\otimes \tau$} (3);
        \draw[->] (1) to[bend right=20] node[pos=.5,auto,swap] {$\scriptstyle \rho$}(3);
      \end{tikzpicture}
    \end{equation*}
for an action $\rho'$ of an ordinary compact group on $X$ and a morphism $\tau:\cC(G)\to A$ of compact quantum groups.   
\end{definition}

\begin{remark}\label{rem.classical}
If $\rho$ is faithful (\Cref{def.faithful}), then classicality is equivalent to $A$ being commutative. This will be the case we are most interested in.

Moreover, $\rho$ is classical if and only if the reduced version $\rho_r$ from \Cref{rem.red} is. Indeed, the surjection $A\to A_r$ is one-to-one on the dense Hopf $*$-subalgebra of $A$ consisting of ``representative functions'' on the quantum group (e.g. \cite[3.1.7]{KusTus99}). 
\end{remark}

Now consider an action $\rho$ as in \Cref{def.classical}. We will need the following notion.

\begin{definition}\label{def.pivot}
A point $x\in X$ is a \define{pivot} for $\rho$ if for every point $y\in X$ and every state $\varphi\in S(A)$ such that $x\triangleleft \varphi$ is a point the probability measure $y\triangleleft\varphi$ only depends on $x\triangleleft \varphi$. 
\end{definition}

The main result of this section says the following:

\begin{proposition}\label{pr.criterion}
An action admitting a pivot is classical. 
\end{proposition}
\begin{proof}
We assume $\rho$ is faithful and try to prove that $A$ is commutative (see \Cref{rem.faithful,rem.classical}). 

Recall the notation $\cM_x$ from the previous section: It was the spectrum of the image of the map \Cref{eq.comp}, and one candidate for the orbit of $x$ under $\rho$. Denote the commutative $C^*$-subalgebra $\cC(\cM_y)\le A$ by $A_y$ for $y\in X$. The pivot hypothesis ensures that the restriction of a state $\varphi$ to $A_y$ is determined by its restriction to $A_x$ so long as this restriction is multiplicative; indeed, this latter multiplicativity requirement is equivalent to $x\triangleleft\varphi$ being a point.  

{\bf Claim: All $A_y$, $y\in X$ are contained in $A_x$.} Suppose $A_y\not\subseteq A_x$ for some $y\in X$ and let $a\in A_y\setminus A_x$ be a positive element.

Define
\[
 L_a = \{\text{ positive }b\in A_x\ |\ b\le a\}
\]
and
\[
 U_a = \{\text{ positive }c\in A_x\ |\ a\le c\}
\]
and let $\psi\in S(A_x)$ be a multiplicative state on $A_x$ for which 
\begin{equation}\label{eq.criterion}
 \sup_{b\in L_a}\psi(b) < \inf_{c\in L_a}\psi(c). 
\end{equation}
To see that such a state exists, note first that the set 
\begin{equation}\label{eq.c-b}
 \{c-b\ |\ c\in U_a,\ b\in L_a\}
\end{equation}
is bounded away from zero in norm (otherwise we could find elements $b\in L_a$ arbitrarily close to $a$ and hence $a$ would belong to $A_x$). Now, if $s>0$ is a lower bound for the norms $\|c-b\|$, then regarding the elements of $A_x$ as continuous functions on $\cM_x$, the sets 
\begin{equation}\label{eq.(c-b)(z)}
 \{z\in \cM_x\ |\ (c-b)(z)\ge s\}
\end{equation}
are all non-empty. 

Note furthermore that $L_a$ is closed under taking suprema and $U_a$ is closed under taking infima with respect to the standard partial order on $A$ induced by the positive elements, and hence \Cref{eq.c-b} is closed under infima. This means that the collection of sets \Cref{eq.(c-b)(z)} is closed under intersections; since all such sets are compact and non-empty, their intersection is also non-empty. Finally, simply take $\psi\in S(A_x)=\pr(\cM_x)$ to be one of the points in this intersection.   

According to \cite[Lemma 2.13]{PauTom09}, for any $t\in\bR$ between the two numbers in \Cref{eq.criterion} $\psi$ can be extended to a positive functional $\varphi$ on all of $A$ such that $\phi(a)=t$. Since $\phi$ is also unital, it must be a state. The range of choices for $t$ means that the restriction to $A_y$ of such an extension is not uniquely determined by the multiplicative state $\psi$, contradicting the pivot condition. This concludes the proof of the claim. 

We now know that the $C^*$-subalgebra of $A$ generated by $A_y$, $y\in X$ is commutative. The faithfulness condition ensures that this $C^*$-algebra is $A$ itself, finishing the proof. 
\end{proof}

\section{Actions with finite orbits are classical}\label{se.fin_orb=>classical}

We now return to the metric setup. Recall that in the setting of isometric actions we can refer unambiguously to orbits in any of their various guises (\Cref{cor.orbits2}). The main result of the section is

\begin{theorem}\label{th.fin_orb=>classical}
An isometric action  with finite orbits on the underlying geodesic metric space of a compact, connected Riemannian manifold is classical. 
\end{theorem}

Starting here and throughout the rest of the paper $(X,d)$ is a compact metric space and $\rho$ an isometric action of a quantum group $A$ on it. The proof of \Cref{th.fin_orb=>classical} proceeds through some auxiliary results. 

For two subsets $Z,Z'\subseteq X$ denote 
\[
 d(Z,Z') = \inf_{z\in S,z'\in T}d(z,z').
\]

\begin{lemma}\label{le.dist_between_orbits}
 For any $x,y\in X$ we have 
\[
 d(\orb_x,\orb_y) = d(\orb_x,y).  
\]
\end{lemma}
\begin{proof}
A priori we have
\[
 d(\orb_x,\orb_y) = \inf_{y'\in\orb_y}d(\orb_x,y'), 
\]
so we will be done if we show that all distances $d(\orb_x,y')$ for $y'\in\orb_y$ are equal.  

Fix some arbitrary $y'\in \orb_y$, and assume $d(\orb_x,y)$ is equal to $d(x,y)$ (it is achieved at some $x'\in \orb_x$ by compactness, and we may as well suppose $x'=x$ since $\orb_{x'}=\orb_x$). Let $\varphi\in S(A_r)$ be a state such that $y'=y\triangleleft\varphi$ (\Cref{cor.orbits2}). 

Now, by \Cref{th.old} there is a $(x\triangleleft\varphi,y\triangleleft \varphi)$-coupling supported on the set of pairs of points that are $r=d(x,y) = d(\orb_x,y)$ apart, so $x\triangleleft \varphi$ must be supported on the sphere of radius $r$ around $y'=y\triangleleft\varphi$. In particular, there are points of $\orb_x$ on that sphere and we have 
\begin{equation*}
 d(\orb_x,y')\le d(\orb_x,y).
\end{equation*}
But $y$ and $y'$ play symmetric roles, so this is actually an equality. Since $y'\in\orb_y$ was arbitrary, we are done.  
\end{proof}

\begin{lemma}\label{le.orbits_open}
For every positive integer $n$, the set $X_{\ge n}$ of points whose orbit has at least $n$ elements is open. 
\end{lemma}
\begin{proof} 
Let $x$ be a point whose orbit has at least $n$ elements, and $y\in X$ any point outside $\orb_x$ (if $\orb_x=X$ there is nothing to prove). 

First, I claim that we may assume without loss of generality that $d(y,\orb_x)=d(y,x)$. The argument for why this is so is very similar to the one from the proof of \Cref{le.dist_between_orbits}. 

By compactness $d(y,\orb_x)=d(y,x')$ for some $x'\in\orb_x$. Now let $\varphi\in S(A_r)$ be a state for which 
\begin{equation}\label{eq:1}
 x'\triangleleft \varphi=x  
\end{equation}
(one exists by \Cref{cor.orbits2}). By \Cref{th.old} there is a coupling $\pi\in\Pi(y\triangleleft\varphi,x'\triangleleft \varphi)$ supported on the set of pairs of points of $X$ that are $d(y,\orb_x)$ apart. This together with \Cref{eq:1} implies that  $y\triangleleft \varphi$ is supported on the sphere of radius $d(y,\orb_x)$ around $x$, and hence that there are points in $\orb_y$ on this sphere. To prove the claim, simply substitute such a point for $y$; its distance from $\orb_x$ is still $d(y,\orb_x)$ by \Cref{le.dist_between_orbits}.  

Now let $x_i$ be $n$ distinct points on $\orb_x$ for $0\le i\le n-1$ with $x_0=x$, and choose $y$ as above so that $d(y,\orb_x)$ is very small (smaller than say half of $\min_{i,j}d(x_i,x_j)$). Just as in the argument above, $n$ states $\varphi_i\in S(A_r)$ such that $x\triangleleft\varphi_i=x_i$ will turn $y$ into measures supported on the spheres of radius $d(y,\orb_x)$ centered respectively at $x_i$. Since these spheres are disjoint, they contain $n$ distinct points in $\orb_y$. 

The conclusion is that every $y$ whose distance from $\orb_x$ is small enough contains has an orbit of cardinality at least $n$, as desired. 
\end{proof}

\begin{lemma}\label{le.orbits_dense}
If $X$ is a compact connected Riemannian manifold and $d$ is its underlying geodesic metric, then the set $X_{\ge n}$ from \Cref{le.orbits_open} is either empty or dense. 
\end{lemma}
\begin{proof}
Assume that $X_{\ge n}$ is non-empty, and let $y\in X$ be arbitrary with $x\in X_{\ge n}$ such that $d(y,\orb_x)=d(y,x)$. 

Connect $y$ and $x$ by a minimizing geodesic $\gamma$, i.e. one whose length $\ell$ is exactly $d(y,x)$. Parametrize $\gamma$ by arc length so that $\gamma(0)=y$ and $\gamma(\ell)=x$ and let $z=\gamma(\varepsilon)$ for some small $\varepsilon>0$. I claim that $z\in X_{\ge n}$. Assuming this for now, we can finish the proof: Any point $y$ has elements $z\in X_{\ge n}$ arbitrarily close to it. 

It remains to prove the claim that $z$ belongs to $X_{\ge n}$. 

Let $x_i$, $0\le i\le n-1$ be $n$ distinct points in $\orb_x$ and $\varphi_i\in S(A_r)$ states such that $x_i=x\triangleleft \varphi_i$. By \Cref{th.old} there are $(y\triangleleft\varphi_i,z\triangleleft\varphi_i)$-couplings supported on the set of pairs of points $\varepsilon$ apart, so we can find $y_i\in\mathrm{supp}(y\triangleleft\varphi_i)$ and $z_i\in \mathrm{supp}(z\triangleleft\varphi_i)$ with $d(y_i,z_i)=\varepsilon$. Moreover, just as in the proofs of \Cref{le.dist_between_orbits,le.orbits_dense} we also have  
\[
 d(y_i,x_i)=\ell,\ d(z_i,x_i) = \ell-\varepsilon. 
\]
The situation is depicted schematically below.
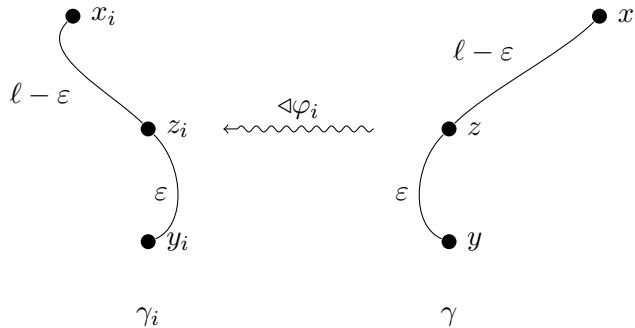
\begin{equation*}
  \begin{tikzpicture}[auto,baseline=(current  bounding  box.center)]
    \draw (0,0) node[circle, inner sep=2pt, fill=black, label={right:{$y$}}] (y) {};
    \draw (2,3) node[circle, inner sep=2pt, fill=black, label={right:{$x$}}] (x) {};
    \draw (0,1.5) node[circle, inner sep=2pt, fill=black, label={right:{$z$}}] (z) {};
    \draw (y)  .. controls +(-.5,.2) and +(-.5,-.5) .. node[pos=.5,auto] {$\varepsilon$} (z);
    \draw (z)  .. controls +(.5,.5) and +(-.5,-.5) .. node[pos=.5,auto] {$\ell-\varepsilon$}  (x);
    \draw [->,decorate,decoration={snake,amplitude=.4mm,segment length=2mm,post length=1mm}]
        (-1,1.5) -- (-3,1.5)
	node [above,align=center, midway]
	{$\triangleleft\varphi_i$};
    \draw (-4,0) node[circle, inner sep=2pt, fill=black, label={right:{$y_i$}}] (yi) {};
    \draw (-5,3) node[circle, inner sep=2pt, fill=black, label={right:{$x_i$}}] (xi) {};
    \draw (-4,1.5) node[circle, inner sep=2pt, fill=black, label={right:{$z_i$}}] (zi) {};
    \draw (yi)  .. controls +(.5,.2) and +(.5,-.5) .. node[pos=.5,auto] {$\varepsilon$} (zi);
    \draw (zi)  .. controls +(-.5,.5) and +(-.5,-.5) .. node[pos=.5,auto] {$\ell-\varepsilon$}  (xi);    
    \node () at (-4,-1) {$\gamma_i$};
    \node () at (0,-1) {$\gamma$};
  \end{tikzpicture}
\end{equation*}
All of the arcs depict minimizing geodesics constructed as follows: For each $i$ connect $y_i$ to $z_i$ by a minimizing geodesic, as well as $z_i$ to $x_i$. The two geodesics continue one another at $z_i$ (we will say that they are collinear): The composite path $\gamma_i$ realizes the minimal distance $\ell$ between $y_i$ and $x_i$ and so it must itself be a geodesic.

It remains to argue that $z_i$ are distinct. Indeed, if $z_i=z_j$ for $i\ne j$ then $\gamma_i$ and $\gamma_j$ intersect at this common point, but they cannot coincide because $x_i\ne x_j$. It follows that the length-$\varepsilon$ segment of $\gamma_i$ from $y_i$ to $z_i=z_j$ and the length-$(\ell-\varepsilon)$ segment of $\gamma_j$ from $z_j=z_i$ to $x_j$ are not collinear: 
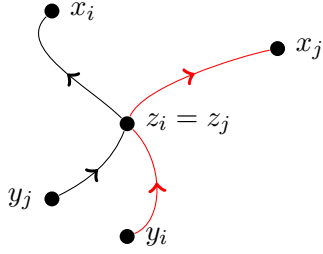
\begin{equation*}
  \begin{tikzpicture}[auto,baseline=(current  bounding  box.center)]
    \draw (0,0) node[circle, inner sep=2pt, fill=black, label={right:{$y_i$}}] (yi) {};
    \draw (-1,3) node[circle, inner sep=2pt, fill=black, label={right:{$x_i$}}] (xi) {};
    \draw (0,1.5) node[circle, inner sep=2pt, fill=black, label={right:{$z_i=z_j$}}] (zi) {};
    \draw[red,decoration={markings, mark=at position .5 with {\arrow[line width=.5mm]{>}}}, postaction={decorate}] (yi)  .. controls +(.5,.2) and +(.5,-.5) .. (zi);
    \draw[decoration={markings, mark=at position .5 with {\arrow[line width=.5mm]{>}}}, postaction={decorate}] (zi)  .. controls +(-.5,.5) and +(-.5,-.5) .. (xi);
    \draw (-1,.5) node[circle, inner sep=2pt, fill=black, label={left:{$y_j$}}] (yj) {};
    \draw (2,2.5) node[circle, inner sep=2pt, fill=black, label={right:{$x_j$}}] (xj) {};  
    \draw[decoration={markings, mark=at position .5 with {\arrow[line width=.5mm]{>}}}, postaction={decorate}] (yj)  .. controls +(.5,.25) and +(-.2,-.5) .. (zi);
    \draw[red,decoration={markings, mark=at position .5 with {\arrow[line width=.5mm]{>}}}, postaction={decorate}] (zi)  .. controls +(.2,.5) and +(-.5,-.1) .. (xj);
  \end{tikzpicture}
\end{equation*}
The composite of these two non-collinear geodesic arcs (red path in the picture above) cannot be minimizing from $y_i$ to $x_j$, so $d(y_i,x_j)$ is strictly smaller than $\ell$. This contradicts the fact that $d(\orb_y,\orb_x)$ is equal to $\ell=d(y,x)=d(y,\orb_x)$ by \Cref{le.dist_between_orbits}.
\end{proof}

We henceforth specialize to the setup from \Cref{le.orbits_dense}: $(X,g)$ is a compact, connected Riemannian manifold and $d$ is its geodesic metric.

\begin{corollary}\label{cor.max_card}
 Under these assumptions, if all orbits are finite then they are uniformly bounded in size.  
\end{corollary}
\begin{proof}
  In other words, we want to prove that the set of orbit cardinalities $|\orb_x|$ has a maximum as $x$ ranges over $X$. 

Suppose all $X_{\ge n}$ are non-empty. They are then all open and dense by \Cref{le.orbits_open,le.orbits_dense}, and their intersection is non-empty (and in fact dense) by Baire's Category Theorem. This contradicts the finiteness of all orbits.
\end{proof}

\begin{definition}\label{def.generic}
 Under the hyptheses of \Cref{cor.max_card} a \define{generic} point $x\in X$ is one whose orbit has maximal cardinality.  
\end{definition}

The following result will be key to the proof of \Cref{th.fin_orb=>classical}.

\begin{lemma}\label{le.key}
  Under the hypotheses of \Cref{th.fin_orb=>classical} let $x\in X$ be a generic point and $\varphi\in S(A_r)$ be a state such that $x\triangleleft \varphi=x$. Then, $y\triangleleft\varphi=y$ for all points $y\in X$. 
\end{lemma}
\begin{proof}
Let $n=|\orb_x|$ and denote the elements of the orbit by $x_i$, $0\le i\le n-1$ with $x_0=x$. Fix an arbitrary point $y\in X$, and let $\gamma$ be a minimizing geodesic arc from $x$ to $y$. We have to show that the set
\begin{equation}\label{eq.S}
 S=\{z\in\gamma\ |\ \text{ all points on the segment }[x,z]\subseteq\gamma\text{ are fixed by }\triangleleft\varphi\} 
\end{equation}
comprises all of $\gamma$. Since $\gamma$ is connected, it is enough to show that $S$ is both open and closed (it is non-empty, since $x\in S$). 

Closure is immediate since $\triangleleft\varphi$ is a continuous self-map of $\pr(X)$ with respect to the usual weak$^*$ topology on the latter set regarded as a subset of $\cC(X)^*$.  

To see that $S$ is also open, let $z\in S$ be a point different from $y$ (if $z=y$ there is nothing to prove). We split the remainder of the proof into two cases. 

{\bf Case 1: $z=x$.} Let $z'\in\gamma$ be very close to $x=z$;(closer, say, than half the number $\min_{x'\ne x''\in\orb_x}d(x',x'')$. 

 $\triangleleft\varphi$ sends $z'$ to probability measure supported on the sphere of radius $d(x,z')$ centered at $x$. If this probability measure is not $z'$ itself, then there must be at least one other point from $\orb_{z'}$ on this sphere. 

Just as in the proof of \Cref{le.orbits_open}, states $\varphi_i$ that send $x$ to $x_i$, $0\le i\le n-1$ of $\orb_x$ will turn $z$ into measures supported on the spheres $S_i$ of radius $d(x,z')$ around $x_i$ respectively. Since $d(x,z')$ is so small that these spheres do not intersect, there is at least one element of $\orb_{z'}$ for each sphere. But there are $n$ spheres and by assumption $n$ is the maximal cardinality of an orbit, so there is \define{exactly} one element of $\orb_{z'}$ on each $S_i$. 

Now, $\triangleleft\varphi$ fixes $x$ and so as before, sends $z'$ to a measure supported on $S_0$ (the sphere of radius $d(x,z')$ centered at $x$). We have just established however that the orbit of $z'$ intersects $S_0$ in a single point, so $z'\triangleleft\varphi=z'$. This applies to any $z'$ so long as it is close enough to $z=x$, finishing the proof of Case 1. 

{\bf Case 2: $z$ is in the interior of $\gamma$.} Let $x'$ and $z'$ be point in the intervals $(x,z)$ and $(z,y)$ of $\gamma$ respectively and very close to $z$. 
 
By assumption, both $x'$ and $z$ are fixed by $\triangleleft\varphi$. This implies that $z'\triangleleft\varphi$ is supported both on the geodesic sphere $S_1$ of radius $d(x',z')$ around $x'$ and the sphere $S_2$ of radius $d(z,z')$ around $z$ (see the figure below). 
  
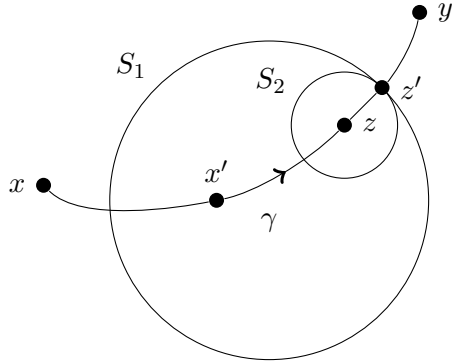
\begin{equation*}
    \begin{tikzpicture}[auto,baseline=(current  bounding  box.center)]
      \coordinate (1) at (0,0);
      \coordinate (2) at (1,1);
      \draw (-3,.2) node[circle, inner sep=2pt, fill=black, label={left:{$x$}}] (x) {};      
      \draw (-.7,0) node[circle, inner sep=2pt, fill=black, label={90:{$x'$}}] (x') {};
      \draw (1,1) node[circle, inner sep=2pt, fill=black, label={right:{$z$}}] (z) {};
      \draw (1.5,1.5) node[circle, inner sep=2pt, fill=black, label={right:{$z'$}}] (z') {};
      \draw (2,2.5) node[circle, inner sep=2pt, fill=black, label={right:{$y$}}] (y) {};
      \node (circ1) [label=135:$S_1$,draw,circle through=(z')] at (1) {};
      \node (circ2) [label=155:$S_2$,draw,circle through=(z')] at (2) {};
      \draw[decoration={markings, mark=at position .5 with {\arrow[line width=.5mm]{>}}}, postaction={decorate}] (x')  .. controls +(.5,.1) and +(-.5,-.5) .. (z);
      \draw[] (z)  .. controls +(.5,.5) and +(-.2,-.2) .. (z');
      \draw[] (x)  .. controls +(.5,-.5) and +(-.5,-.1) .. (x');
      \draw[] (z')  .. controls +(.1,.1) and +(-.1,-.5) .. (y);
      \node () at (0,-.3) {$\gamma$};
    \end{tikzpicture}
  \end{equation*}
If $x'$ and $z'$ are close enough to $z$ though (for example so that $d(x',z')$ is smaller than the injectivity radius of $(X,g)$) these two spheres only intersect in $z'$, and hence $z'\triangleleft\varphi=z'$. This applies to all $z'\in (z,y)$, showing that indeed $z$ is in the interior of the set \Cref{eq.S}. 
\end{proof}

\begin{proof_of_fin_orb}
Throughout the proof we will assume that $A$ is reduced (passing from $\rho$ to $\rho_r$ does not affect the classical character of an action; see \Cref{rem.classical}). By \Cref{pr.criterion}, we will be done if we prove that every point of $X$ that is generic in the sense of \Cref{def.generic} is a pivot. Consequently, this is now the goal. 

Let $x\in X$ be generic. We have to show that if $\psi$ and $\varphi$ are states on $A$ such that $x'=x\triangleleft\psi=x\triangleleft\varphi$ is a point, then $y\triangleleft\psi=y\triangleleft\varphi$ for every $y\in X$ (this is what it means for $x$ to be a pivot; see \Cref{def.pivot}).

Recall the notation $\ol{\varphi}=\varphi\circ\kappa$ (where $\kappa:A\to A$ was the antipode). We know from \Cref{cor.orbits} that $x\triangleleft\left(\psi\ol{\varphi}\right)=x$, so by \Cref{le.key} $y\triangleleft\left(\psi\ol{\varphi}\right)=y$ for all $y\in X$. Applying $\triangleleft\varphi$ to this last equality we get 
\begin{equation}\label{eq.proof_of_fin_orb}
 y\triangleleft\left(\psi\ol{\varphi}\varphi\right) = y\triangleleft\varphi.
\end{equation}
{\bf Claim: $y\triangleleft \psi$ is a point.} Indeed, otherwise it would be a non-trivial (perhaps continuous, i.e. infinite) convex combination of points. Since $\triangleleft\ol{\varphi}$ preserves convex combinations and cannot compress a non-Dirac measure to a single point (e.g. because it cannot send distinct points to the same point by \Cref{cor.orbits}), 
\[
 y=y\triangleleft\left(\psi\ol{\varphi}\right) = (y\triangleleft\psi)\triangleleft\ol{\varphi}
\]
could not be a point. This proves the claim. It applies equally well to the state $\ol{\varphi}$, since $x=x'\triangleleft\ol{\varphi}=x\triangleleft\ol{\psi}$ (so that we could substitute $x'$, $\ol{\varphi}$, $\ol{\psi}$ for $x$, $\psi$ and $\varphi$ respectively). 

This now means that $\triangleleft\ol{\varphi}$ turns points into points, so by \Cref{cor.orbits} $\triangleleft\left(\ol{\varphi}\varphi\right)$ is the identity on $X$. But then the left hand side of \Cref{eq.proof_of_fin_orb} is $y\triangleleft\psi$ and we get the desired conclusion: $y\triangleleft\psi=y\triangleleft\varphi$. 
\end{proof_of_fin_orb}

\section{Negative curvature and rigidity}\label{se.neg_curv}

The current standing assumptions are as follows: $(X,g)$ is a compact connected Riemannian manifold, $d$ is its geodesic metric, and $\rho$ is an isometric action of a quantum group $A$ on $(X,d)$. Moreover, we will henceforth assume that $A$ is reduced. As observed before (e.g. in \Cref{rem.classical} or the proof of \Cref{th.fin_orb=>classical}) this is harmless for the purpose of showing that an action is classical.  

We will denote the closed ball of radius $r$ centered at $x$ by $B(x,r)$ and the corresponding sphere by $S(x,r)$. For background on Riemannian geometry in general and negatively curved manifolds in particular we refer variously to \cite{Bal95,BGS,doC92,GHL}.

We can now state the main result of the section.

\begin{theorem}\label{th.main}
 If $(X,g)$ has negative sectional curvature then $\rho$ has finite orbits.   
\end{theorem}

In particular, by \Cref{th.fin_orb=>classical} we get

\begin{corollary}\label{cor.main}
 Under the hypotheses of \Cref{th.main} the action $\rho$ is classical. 
\qedhere 
\end{corollary}

We will again work through a series of preparatory results, but we need some notation. 

Let $\mu$ be a probability measure supported on the geodesic sphere of radius $r$ around $x\in X$ and assume that $r$ is small enough (this will typically mean smaller than the injectivity radius of $(X,g)$). We will talk about the geodesic connecting $x$ and $\mu$; this concept extends that of a geodesic connecting two points, and is defined as follows:

Connect $x$ to the points $y$ in the support of $\mu$ via arclength-parametrized geodesics $\gamma=\gamma_{x,y}$ respectively (so that $\gamma(0)=x$ and $\gamma(r)=y$). The smallness assumption on $r$ ensures that there is at most one way to do this for each $y$. 

Now simply ``flow'' $\mu$ backwards along these geodesics to get probability measures $\mu_{x,s}$ supported on spheres of radius $s\le r$ centered at $x$. The process can even be continued at negative $t$ to get probability measures $\mu_{x,t}$ supported on spheres of radius $|t|=-t$ ``on the other side'' of $x$ as compared to $\mu$. The following picture briefly describes the situation,  

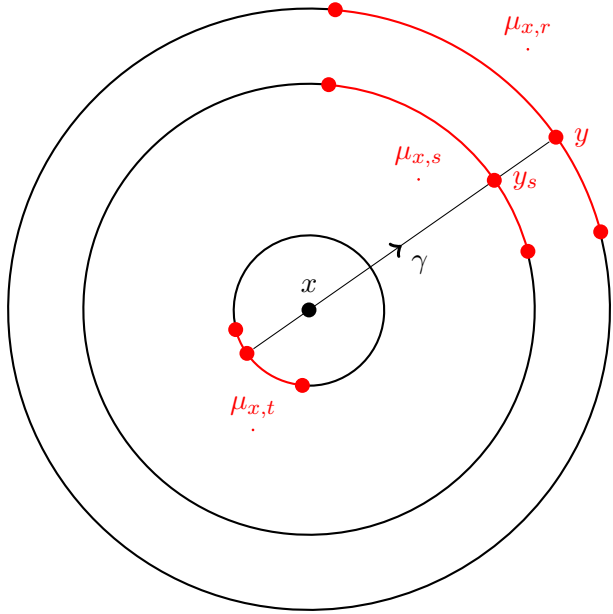
\begin{equation*}
  \begin{tikzpicture}[auto,baseline=(current  bounding  box.center)]
    \coordinate (center) at (0,0);
    \draw (center) node[circle, inner sep=2pt, fill=black, label={north:{$x$}}] (x) {};

    \node[circle,inner sep=0,minimum size={8cm}](r) at (center) {};
    \draw[red,thick] (r.15) arc (15:85:4cm);
    \draw[black,thick] (r.85) arc (85:375:4cm);
    \draw (r.15) node[circle, inner sep=2pt, fill=red, label={right:{}}] () {};
    \draw (r.85) node[circle, inner sep=2pt, fill=red, label={right:{}}] () {};

    \node[circle,inner sep=0,minimum size={6cm}](s) at (center) {};
    \draw[red,thick] (s.15) arc (15:85:3cm);
    \draw[black,thick] (s.85) arc (85:375:3cm);
    \draw (s.15) node[circle, inner sep=2pt, fill=red, label={right:{}}] () {};
    \draw (s.85) node[circle, inner sep=2pt, fill=red, label={right:{}}] () {};

    \node[circle,inner sep=0,minimum size={2cm}](t) at (center) {};
    \draw[red,thick] (t.195) arc (195:265:1cm);
    \draw[black,thick] (t.265) arc (-95:195:1cm);
    \draw (t.195) node[circle, inner sep=2pt, fill=red, label={right:{}}] () {};
    \draw (t.265) node[circle, inner sep=2pt, fill=red, label={right:{}}] () {};

    \draw (r.35) node[circle, inner sep=2pt, fill=red, label={[red]right:{$y$}}] (y) {};

    \draw (t.215) node[circle, inner sep=2pt, fill=red, label={[red]right:{}}] (yt) {};

    \draw[decoration={markings, mark=at position .5 with {\arrow[line width=.5mm]{>}}}, postaction={decorate}] (yt) to node[pos=.5,auto,swap] {$\gamma$} (y);

    \draw (s.35) node[circle, inner sep=2pt, fill=red, label={[red]right:{$y_s$}}] (ys) {};

    \node[circle,inner sep=0,minimum size={4.5cm}](s<r) at (center) {};

    \draw (s<r.50) node[circle, inner sep=0pt, fill=red, label={[red]north:{$\mu_{x,s}$}}] (muxs) {};

    \node[circle,inner sep=0,minimum size={9cm}](r<) at (center) {};

    \draw (r<.50) node[circle, inner sep=0pt, fill=red, label={[red]north:{$\mu_{x,r}$}}] (muxs) {};

    \node[circle,inner sep=0,minimum size={3.5cm}](<0) at (center) {};

    \draw (<0.-115) node[circle, inner sep=0pt, fill=red, label={[red]north:{$\mu_{x,t}$}}] (muxs) {};

  \end{tikzpicture}
\end{equation*}
where the red arcs represent the supports of the respective measures $\mu$, $s>0$ and $t<0$ are real numbers, and in general, for $u\in\bR$, $y_u$ stands for $\gamma(u)$ with $\gamma=\gamma_{x,y}$ as in the discussion above. The measures $\mu_{x,s}$ for $s$ ranging over say $[-(r+\varepsilon),r+\varepsilon]$ for small $\varepsilon>0$ can be thought of as points of a geodesic connecting $x$ and $\mu=\mu_{x,r}$.

In the next result we keep the notation we have just used: $x,y\in X$ will be points such that $r=d(x,y)$ is smaller than the injectivity radius, and $\gamma$ is the arclength-parametrized geodesic from $x$ to $y$. We set $y_u=\gamma(u)$.

\begin{lemma}\label{le.r+epsilon}
If $\varphi\in S(A)$ is such that $x'=x\triangleleft\varphi$ is a point, then denoting $\mu=\mu_{x',r}=y\triangleleft\varphi$ we have 
\begin{equation*}
  y_s\triangleleft\varphi = \mu_{x',s}
\end{equation*}
for all $s\in [-\varepsilon,r+\varepsilon]$ for $\varepsilon>0$ such that $r+\varepsilon$ is smaller than the injectivity radius of $X$. 
\end{lemma}

\begin{remark}
In other words, $\triangleleft\varphi$ turns a small portion of the geodesic connecting $x$ and $y$ into a small portion of the geodesic connecting $x$ and $\mu$ (see the discussion preceding the statement of the lemma).   
\end{remark}

\begin{proof}
We split it into several cases.

{\bf Case 1: $s<0$.} If $|s|=-s$ is small enough (which we assume it is), then $d(x,y_s)=-s$ and $d(y_s,y_r)=r-s$. 

As in the proofs of \Cref{le.orbits_open,le.key}, $\triangleleft\varphi$ turns $y_s$ and $y=y_r$ into measures $\nu$ and $\mu$ supported on the spheres centered at $x'$ of radii $-s$ and $r$ respectively. At the same time though, \Cref{th.old} shows that there is a $(\mu,\nu)$-coupling supported on the set of pairs of points in $X$ that are $r-s$ apart. 

Now, the smallness of $-s$ and $r$ ensure that in the following picture, for every point $z'$ on $S(x',r)$ the green spheres $S(z',r-s)$ and $S(x',-s)$ intersect in a single point, namely $z'_s=\gamma_{x',z'}(s)$ (where as before, $\gamma_{x',z'}$ is the arclength-parametrized geodesic from $x'$ to $z'$).  
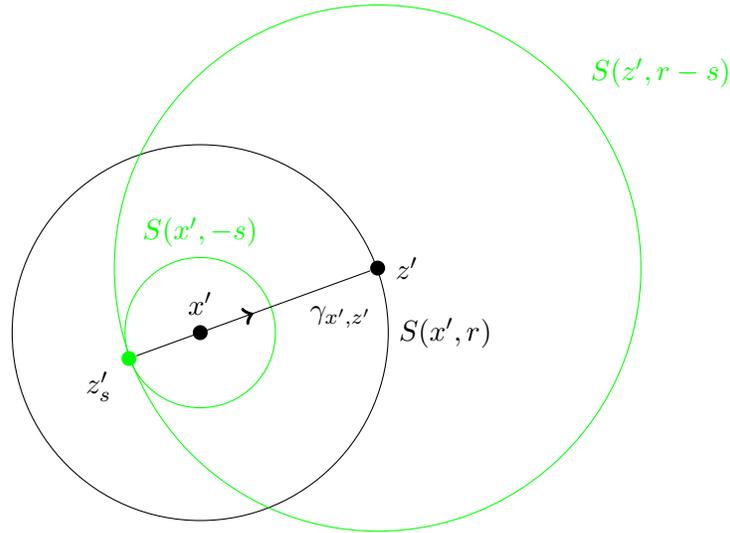
\begin{equation*}
  \begin{tikzpicture}[auto,baseline=(current  bounding  box.center)]
    \coordinate (center) at (0,0);
    \draw (center) node[circle, inner sep=2pt, fill=black, label={north:{$x'$}}] (x') {};

    \node[circle,inner sep=0,minimum size={5cm},label={right:{$S(x',r)$}}](r) at (center) {};
    
    \draw[] (center) circle (2.5cm);

    \node[circle,inner sep=0,minimum size={2cm},label={[green]north:{$S(x',-s)$}}](s) at (center) {};
    
    \draw[green] (x') circle (1cm);

    \draw (r.20) node[circle, inner sep=2pt, fill=black, label={right:{$z'$}}] (z') {};

    \draw[green] (z') circle (3.5cm);

    \node[circle,inner sep=0,minimum size={7cm},label={[green]40:{$S(z',r-s)$}}](r-s) at (z') {};

    \draw (s.200) node[circle, inner sep=2pt, fill=green, label={200:{$z'_s$}}] (z's) {};

    \draw[decoration={markings, mark=at position .5 with {\arrow[line width=.5mm]{>}}}, postaction={decorate}] (z's) to node[pos=.7,auto,swap] {$\gamma_{x',z'}$} (z');

  \end{tikzpicture}
\end{equation*}
The conclusion follows from this: Every point $z'\in S(x',r)$ determines a unique point $z'_s\in S(x',-s)$ whose distance from it is $r-s$, and so the only measure $\nu$ supported on $S(x',-s)$ that could possibly admit a $(\mu,\nu)$-coupling supported on pairs of points $r-s$ away from one another is the flow of $\mu$ along geodesics through $x'$, as described in the discussion preceding this lemma.

{\bf Case 2: $0<s<r$.} This is very similar to Case 1, so we will not argue in any detail; suffice it to say that this time around the relevant picture is
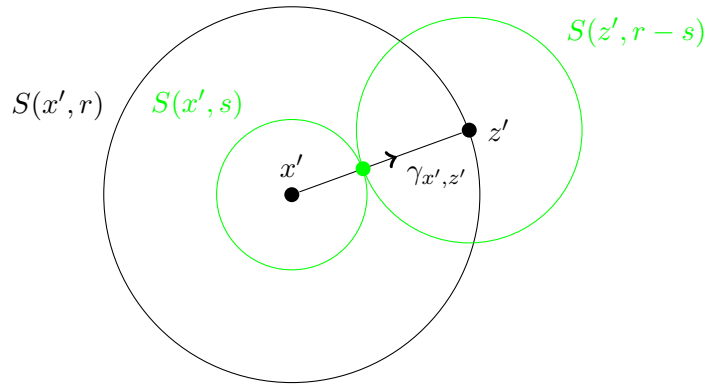
\begin{equation*}
  \begin{tikzpicture}[auto,baseline=(current  bounding  box.center)]
    \coordinate (center) at (0,0);
    \draw (center) node[circle, inner sep=2pt, fill=black, label={north:{$x'$}}] (x') {};

    \node[circle,inner sep=0,minimum size={5cm},label={160:{$S(x',r)$}}](r) at (center) {};
    
    \draw[] (center) circle (2.5cm);

    \node[circle,inner sep=0,minimum size={2cm},label={[green]120:{$S(x',s)$}}](s) at (center) {};
    
    \draw[green] (x') circle (1cm);

    \draw (r.20) node[circle, inner sep=2pt, fill=black, label={right:{$z'$}}] (z') {};

    \draw[green] (z') circle (1.5cm);

    \node[circle,inner sep=0,minimum size={3cm},label={[green]40:{$S(z',r-s)$}}](r-s) at (z') {};

    \draw[decoration={markings, mark=at position .6 with {\arrow[line width=.5mm]{>}}}, postaction={decorate}] (x') to node[pos=.6,auto,swap] {$\gamma_{x',z'}$} (z');

    \draw (s.20) node[circle, inner sep=2pt, fill=green, label={200:{}}] (z's) {};

  \end{tikzpicture}
\end{equation*}
where once more the green point where the green spheres intersect is $z'_s=\gamma_{x',z'}(s)$.

{\bf Case 3: $r<s$.} This is again entirely analogous, the picture 
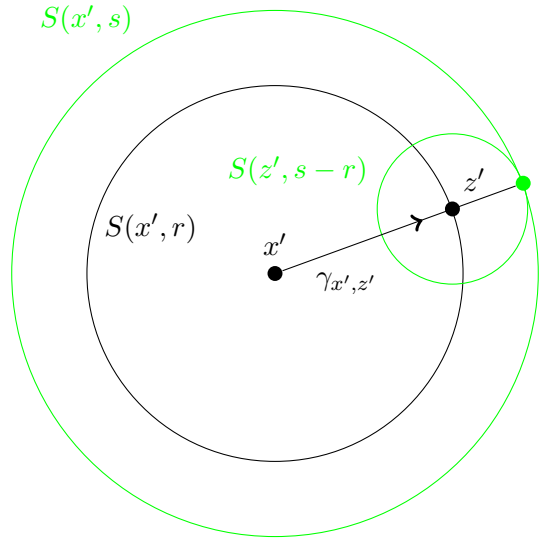
\begin{equation*}
  \begin{tikzpicture}[auto,baseline=(current  bounding  box.center)]
    \coordinate (center) at (0,0);
    \draw (center) node[circle, inner sep=2pt, fill=black, label={north:{$x'$}}] (x') {};

    \node[circle,inner sep=0,text depth=3.5cm,minimum size={5cm},label={}](r) at (center) {};

    \node[circle,inner sep=0,minimum size={7cm},label={[green]120:{$S(x',s)$}}](s) at (center) {};    
    
    \draw[] (center) circle (2.5cm);

    \node[circle,inner sep=0,minimum size={3.5cm}](circaux) at (center) {};

    \node[circle,inner sep=0]() at (circaux.160) {$S(x',r)$};
    
    \draw[green] (x') circle (3.5cm);

    \draw (r.20) node[circle, inner sep=2pt, fill=black, label={85:{$z'$}}] (z') {};

    \draw[green] (z') circle (1cm);

    \node[circle,inner sep=0,minimum size={2cm},label={[green]170:{$S(z',s-r)$}}](s-r) at (z') {};

    \draw (s.20) node[circle, inner sep=2pt, fill=green, label={200:{}}] (z's) {};

    \draw[decoration={markings, mark=at position .6 with {\arrow[line width=.5mm]{>}}}, postaction={decorate}] (x') to node[pos=.1,auto,swap] {$\gamma_{x',z'}$} (z's);

  \end{tikzpicture}
\end{equation*}
being a schematic description of the situation. 
\end{proof}

\begin{lemma}\label{le.close}
If the distance between $y,z\in X$ is smaller than the injectivity radius of $X$ and $\orb_y$ and $\orb_z$ are finite, then every point on the geodesic connecting $y$ and $z$ has finite orbit.     
\end{lemma}

\begin{remark}\label{rem.close}
The conclusion of the lemma refers to the entire geodesic, extended indefinitely in both directions, not just the geodesic segment $[y,z]$. 
\end{remark}

\begin{proof}
Let $r=d(y,z)$ and denote by $\gamma$ the arclength-parametrized geodesic starting out from $y$ towards $z$. We divide the proof into several parts. 

{\bf Step 1: The geodesic segment $[y,z]\subset\gamma$.} Fix a point $x$ on this segment and let $x'=x\triangleleft\varphi$ be a point on the orbit $\orb_x$ for $\varphi\in S(A)$ (all points on the $x$-orbit are of this form for some $\varphi$; see e.g. \Cref{cor.orbits2}).

Now, \Cref{le.r+epsilon} applied to $x$ and $x'=x\triangleleft\varphi$ shows that $\mu=y\triangleleft\varphi$ and $\nu=z\triangleleft\varphi$ are obtained from one another by flowing along geodesics through $x'$, as sketched in the picture

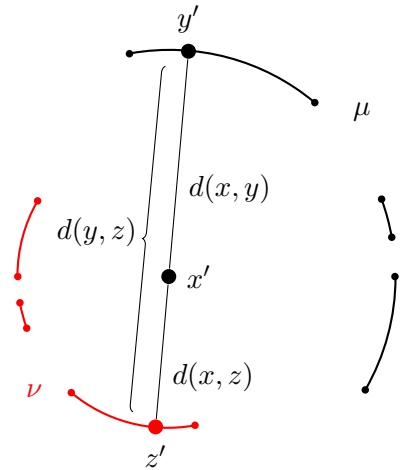
\begin{equation*}
  \begin{tikzpicture}[auto,baseline=(current  bounding  box.center)]
    \coordinate (center) at (0,0);
    \draw (center) node[circle, inner sep=2pt, fill=black, label={right:{$x'$}}] (x') {};


    \node[circle,inner sep=0,minimum size={6cm},label={40:{$\mu$}}](mu) at (center) {};

    \draw (mu.10) node[circle, inner sep=1pt, fill=black] () {};
    \draw (mu.20) node[circle, inner sep=1pt, fill=black] () {};
    \draw[black,thick] (mu.10) arc (10:20:3cm);

    \draw (mu.-30) node[circle, inner sep=1pt, fill=black] () {};
    \draw (mu.0) node[circle, inner sep=1pt, fill=black] () {};
    \draw[black,thick] (mu.-30) arc (-30:0:3cm);

    \draw (mu.50) node[circle, inner sep=1pt, fill=black] () {};
    \draw (mu.100) node[circle, inner sep=1pt, fill=black] () {};
    \draw[black,thick] (mu.50) arc (50:100:3cm);
    

    \node[circle,inner sep=0,minimum size={4cm},label={[red]220:{$\nu$}}](nu) at (center) {};

    \draw (nu.190) node[circle, inner sep=1pt, fill=red] () {};
    \draw (nu.200) node[circle, inner sep=1pt, fill=red] () {};
    \draw[red,thick] (nu.190) arc (190:200:2cm);

    \draw (nu.150) node[circle, inner sep=1pt, fill=red] () {};
    \draw (nu.180) node[circle, inner sep=1pt, fill=red] () {};
    \draw[red,thick] (nu.150) arc (150:180:2cm);

    \draw (nu.230) node[circle, inner sep=1pt, fill=red] () {};
    \draw (nu.280) node[circle, inner sep=1pt, fill=red] () {};
    \draw[red,thick] (nu.230) arc (230:280:2cm);
    

    \draw (mu.85) node[circle, inner sep=2pt, fill=black,label={above:{$y'$}}] (y') {};

    \draw (nu.265) node[circle, inner sep=2pt, fill=red,label={below:{$z'$}}] (z') {};

    \draw (x') to node[pos=.5,auto,swap] {$d(x,y)$} (y');
    \draw (x') to node[pos=.5,auto] {$d(x,z)$} (z');
    \draw[decoration={brace,mirror},decorate] ([xshift=-.2cm,yshift=-.2cm]y'.west) to node[pos=.55,auto,swap] {$d(y,z)$} ([xshift=-.2cm,yshift=.2cm]z'.west);

  \end{tikzpicture}
\end{equation*}
(the black and red arcs representing the supports of $\mu$ and $\nu$ respectively). In particular, $x'$ sits on at least one geodesic segment of length $r$ between points $y'\in \orb_y$ and $z'\in \orb_z$ such that $d(x',y')=d(x,y)$ and $d(x',z')=d(x,z)$. By assumption, there are only finitely many choices for the endpoints $x',y'$ of these segments, and hence only finitely many segments (because we are working at small enough length scales). Finally, this shows that there are only finitely many $z'$. 

{\bf Step 2: A geodesic segment of $\gamma$ slightly larger than $[y,z]$.} This is very similar to Step 1. Note that in the proof of the latter we used Case 1 of the proof of \Cref{le.r+epsilon}. This time around we will use Case 2 (or 3). 

Since the situation is perfectly symmetric in $y$ and $z$ we fix without loss of generality an $x$ on $\gamma$ such that $z$ is in the interior of the geodesic arc $[x,y]$.

As before, let $x'=x\triangleleft\varphi$ be a point in $\orb_x$. If $d(x,y)$ is smaller than the injectivity radius, then once more \Cref{le.r+epsilon} shows that $\nu=z\triangleleft\varphi$ is obtained from $\mu=y\triangleleft\varphi$ by flowing along geodesics through $x$:
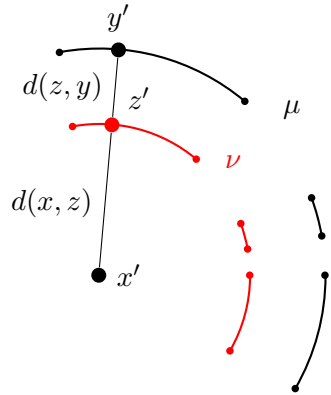
\begin{equation*}
  \begin{tikzpicture}[auto,baseline=(current  bounding  box.center)]
    \coordinate (center) at (0,0);
    \draw (center) node[circle, inner sep=2pt, fill=black, label={right:{$x'$}}] (x') {};


    \node[circle,inner sep=0,minimum size={6cm},label={40:{$\mu$}}](mu) at (center) {};

    \draw (mu.10) node[circle, inner sep=1pt, fill=black] () {};
    \draw (mu.20) node[circle, inner sep=1pt, fill=black] () {};
    \draw[black,thick] (mu.10) arc (10:20:3cm);

    \draw (mu.-30) node[circle, inner sep=1pt, fill=black] () {};
    \draw (mu.0) node[circle, inner sep=1pt, fill=black] () {};
    \draw[black,thick] (mu.-30) arc (-30:0:3cm);

    \draw (mu.50) node[circle, inner sep=1pt, fill=black] () {};
    \draw (mu.100) node[circle, inner sep=1pt, fill=black] () {};
    \draw[black,thick] (mu.50) arc (50:100:3cm);
    

    \node[circle,inner sep=0,minimum size={4cm},label={[red]40:{$\nu$}}](nu) at (center) {};

    \draw (nu.10) node[circle, inner sep=1pt, fill=red] () {};
    \draw (nu.20) node[circle, inner sep=1pt, fill=red] () {};
    \draw[red,thick] (nu.10) arc (10:20:2cm);

    \draw (nu.-30) node[circle, inner sep=1pt, fill=red] () {};
    \draw (nu.0) node[circle, inner sep=1pt, fill=red] () {};
    \draw[red,thick] (nu.-30) arc (-30:0:2cm);

    \draw (nu.50) node[circle, inner sep=1pt, fill=red] () {};
    \draw (nu.100) node[circle, inner sep=1pt, fill=red] () {};
    \draw[red,thick] (nu.50) arc (50:100:2cm);
    

    \draw (mu.85) node[circle, inner sep=2pt, fill=black,label={above:{$y'$}}] (y') {};

    \draw (nu.85) node[circle, inner sep=2pt, fill=red,label={45:{$z'$}}] (z') {};

    \draw (z') to node[pos=.1,auto] {$d(z,y)$} (y');
    \draw (x') to node[pos=.3,auto] {$d(x,z)$} (z');
%
  \end{tikzpicture}
\end{equation*}
We conclude as in Step 1. 

{\bf Step 3: All of $\gamma$.}  An examination of the proof of Step 2 reveals that it extends the conclusion of the theorem to points on a geodesic segment of length $r+\varepsilon$ for some $\varepsilon>0$ depending only on $r$. We can keep applying this procedure to segments further and further out towards infinity, like say $[\gamma(n\varepsilon),\gamma(r+n\varepsilon)]$ for integers $n$. These will cover all of $\gamma$.    
\end{proof}

Before stating the next result we recall some terminology.

\begin{definition}\label{def.defns}
An embedded submanifold $M$ of $X$ is \define{totally geodesic} (\cite[Definition 2.80 bis]{GHL} or \cite[p. 132]{doC92}) if for every point $m\in M$ and every tangent vector to $M$ at $m$ the geodesic in $X$ in the direction of $v$ lies entirely within $M$.   

A subset $S$ of $X$ is \define{convex} if every two points in the closure of $S$ are joined by a unique minimizing geodesic segment in $X$, and the interior of that segment is contained in $S$. 

$S\subseteq X$ is \define{locally convex} if every point $s\in S$ has a convex neighborhood in $S$.    
\end{definition}

\begin{remark}
The notion of convexity in \Cref{def.defns} is called \define{strong convexity} in \cite[Chapter 3, $\S$ 4]{doC92}. 
\end{remark}

\begin{proposition}\label{pr.tot_geo}
The set $X_{\cat{fin}}$ of points of $X$ having finite orbit under $\rho$ is a closed, totally geodesic, locally convex submanifold.   
\end{proposition}
\begin{proof}
Let $B\subset X$ be a small convex open ball. Since $X$ can be covered with such balls \cite[Chapter 3, Proposition 4.2]{doC92}, it suffices to show that the intersection $S=X_{\cat{fin}}\cap B$ is a closed, connected, totally geodesic and convex submanifold of $B$. 

On the other hand, once we prove that $S$ is a closed submanifold of $B$, connected, convexity and the fact that $S$ is totally geodesic will follow: 

Connectedness is immediate from \Cref{le.close}, which implies that any two points in $S$ are connected by some geodesic segment contained in $S$, as is convexity. Finally, the totally geodesic property follows from the closure of $S$ in $B$ and the fact that a geodesic ray tangent to $S$ and emanating from $s\in S$ is a limit of geodesic segments connecting $s$ to other points in $S$, and all such segments are contained in $S$ by \Cref{le.close}. Hence, it remains to show that $S$ is indeed a closed submanifold. 

Now let $n$ be the largest dimension of a submanifold $S'$ of $B$ contained in $S$ and choose an interior point $s$ of $S'$. Consider the preimage $\exp_s^{-1}(S')\subset T_sX$ of $S'$ through the exponential map based at $s$. Since the exponential map is a diffeomorphism in some neighborhood of $s$ in $T_sX$ (say in $\exp_s^{-1}(B)$), this preimage is a submanifold of the tangent space $T_sX$ containing zero. 

{\bf Claim: Some small neighborhood of $s$ in $\exp_s^{-1}(S')$ coincides with a small neighborhood of a linear subspace of $T_sX$.}

Indeed, otherwise the cone in $T_sX$ spanned by $\exp_s^{-1}(S')$ (i.e. the union of the rays passing through the points of the preimage) would contain some $(n+1)$-dimensional submanifold; in the picture below $\exp_s^{-1}S$ is depicted as a curve and the so that $n=1$ and the $(n+1)$-dimensional manifold we referred to just now will be swept by the straight segments based at $0\in T_sX$. 
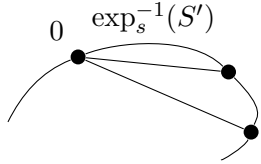
\begin{equation*}
  \begin{tikzpicture}[auto,baseline=(current  bounding  box.center)]
    \draw (0,0) node[circle, inner sep=2pt, fill=black, label={120:{$0$}}] (s) {};
    \draw (2,-.2) node[circle, inner sep=2pt, fill=black, label={}] (1) {};
    \draw (2.3,-1) node[circle, inner sep=2pt, fill=black, label={}] (2) {};
    \node (aux1) at (-1,-1) {};
    \node (aux2) at (1.8,-1.5) {};
    \draw (aux1) .. controls +(.25,.5) and +(-.5,-.2)  .. (s);
    \draw (s) .. controls +(.5,.2) and +(-.5,.5)  .. node[pos=.5,auto] {$ \exp_s^{-1}(S')$} (1);
    \draw (1) .. controls +(.5,-.5) and +(.08,.2)  .. (2);
    \draw (2) .. controls +(-.08,-.2) and +(.08,.1)  .. (aux2);
    \draw (s) to (1);
    \draw (s) to (2);
  \end{tikzpicture}
\end{equation*}
Since the image of this cone through $\exp_s$ is contained in $S$ by \Cref{le.close}, this would contradict the maximality of $n$. This settles the claim. 

We now know that in a neighborhood of $s$ the manifold $S'$ coincides with the image through $\exp_s$ of $T_sS'$. Enlarging $S'$ if necessary, we can assume that
\begin{equation*}
S' = \exp_s(T_sS')\cap B.   
\end{equation*}
{\bf Claim: $S''=S$.} In particular $S$ is a submanifold, which as observed earlier will complete the proof of the proposition. 

The inclusion $S'\subseteq S$ follows from \Cref{le.close} and the previous claim. 

To prove the other inclusion, consider the geodesic segment $\gamma$ from some arbitrary $t\in S$ to $s$. If $\gamma$ is not contained in $S'$, then the tangent to it at $s$ lies outside the subspace $T_sS'\le T_sX$ and the geodesic segments connecting $t$ to points in a small neighborhood of $s$ in $S'$ sweep out an $(n+1)$-dimensional manifold contained in $S$ by \Cref{le.close}. 
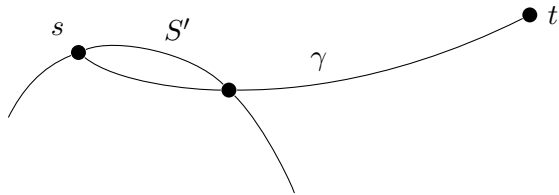
\begin{equation*}
  \begin{tikzpicture}[auto,baseline=(current  bounding  box.center)]
    \draw (0,0) node[circle, inner sep=2pt, fill=black, label={120:{$s$}}] (s) {};
    \draw (2,-.5) node[circle, inner sep=2pt, fill=black, label={}] (s') {};
    \draw (6,.5) node[circle, inner sep=2pt, fill=black, label={right:{$t$}}] (t) {};
    \node (aux1) at (-1,-1) {};
    \node (aux2) at (3,-2) {};
    \draw (s)  .. controls +(.5,-.4) and +(-.5,0) .. (s');
    \draw (s')  .. controls +(.5,0) and +(-2,-1) .. node[pos=.5,auto] {$\gamma$}  (t);
    \draw (aux1) .. controls +(.25,.5) and +(-.5,-.2)  .. (s);
    \draw (s) .. controls +(.5,.2) and +(-.5,.5)  .. node[pos=.5,auto] {$S'$} (s');
    \draw (s') .. controls +(.5,-.5) and +(-.1,.1)  .. (aux2);
  \end{tikzpicture}
\end{equation*}
Once more, this contradicts the maximality of $n$. This shows that as desired, the arbitrary point $t\in S$ is contained in $S''$.   
\end{proof}

We remind the reader that a \define{closed geodesic} in $X$ is a closed curve that is everywhere a geodesic (in other words, a geodesic $\gamma:\bR\to X$ such that $\gamma(s) = \gamma(t)$ for some $s\ne t$). A result of E. Cartan states that for any non-trivial element $a$ of the fundamental group $\pi_1 = \pi_1(X)$ there is a shortest loop in $X$ representing $a$ that is a closed geodesic (\cite[Chapter 12, Theorem 2.2]{doC92}).  

Denote by $\wt{X}$ the universal cover of $X$, equipped with the projection $p:\wt{X}\to X$; it is a complete, connected and simply connected Riemannian manifold. We denote its own geodesic distance by $d$ again, relying on context to distinguish this from the geodesic distance on $X$. 

$\pi_1$ acts by deck transformations on $\wt{X}$; following \cite[$\S$ II.3]{Bal95} or \cite[$\S$ 6.1]{BGS}, for $a\in \pi_1$ and a point $x\in \wt{X}$ we denote by $d_a(x)$ the distance $d(x,ax)$. If $\gamma$ is a closed geodesic in the class of $a\in \pi_1$ as in the discussion above, then $a$ fixes any lift $\wt{\gamma}$ of $\gamma$ in $\wt{X}$ (see the proof of \cite[Chapter 12, Proposition 2.6]{doC92}), and the length of $\gamma$ is exactly 
\begin{equation*}
\inf_{x\in \wt{X}}d_a(x) = \min_{x\in \wt{X}} d_a(x).  
\end{equation*}

Let us now specialize to the situation covered by \Cref{th.main}, when the sectional curvature of $(X,g)$ is everywhere (strictly) negative. In that case, for every non-trivial element of $\pi_1(X)$ the corresponding closed geodesic referred to above is unique (this is essentially the content of \cite[Chapter 12, Lemma 3.3]{doC92}). 

The space of closed geodesics of a given length is compact with respect to the topology induced by the Hausdorff distance on closed subsets of $X$. If two such geodesics are too close to one another, then they must be in the same homotopy class. The uniqueness ensured by the negative curvature condition implies that the space of closed geodesics of a given length is also discrete, and hence finite. In particular, there are only finitely many closed geodesics of minimal length.

\begin{definition}
A closed geodesic in $X$ is \define{minimal} if it has minimal length among all closed geodesics. 

A point $x\in X$ is \define{minimal} if it lies on one of the finitely many minimal closed geodesics. 
\end{definition}

The notion of minimal point will be helpful for us for the following reason.

\begin{lemma}\label{le.min_geo_inv}
Under the hypotheses of \Cref{th.main} the orbit of any minimal point of $X$ consists of minimal points.  
\end{lemma}
\begin{proof}
Let $x$ be a point on the minimal-length closed geodesic $\gamma$, and let $y\in \gamma$ be the point farthest from $x$. In other words, $r=d(x,y)$ is half the length of $\gamma$ and also equal to the injectivity radius of $X$. 

Let $\varphi\in S(A)$ be a state such that $x'=x\triangleleft \varphi$ is a point. As in the proof of \Cref{le.orbits_open}, the probability measure $y\triangleleft\varphi$ is supported on the sphere $S(x',r)$. Moreover, applying \Cref{le.r+epsilon} to both halves of the geodesic $\gamma$ connecting $x$ and $y$ it follows that there are geodesic loops of length $2r$ based at every point in the support of $y\triangleleft\varphi$ (and passing through $x'$). 

By the minimality of the closed geodesic length $2r$, all of the above loops must be closed geodesics (otherwise there would be strictly shorter closed geodesic representatives in their respective homotopy classes, as in the proof of \cite[Chapter 12, Theorem 2.2]{doC92}). 
\end{proof}

\begin{lemma}\label{le.fin_nonempty}
Under the hypotheses of \Cref{th.main} the set $X_{\cat{fin}}\subseteq X$ of points with finite orbit is non-empty. 
\end{lemma}
\begin{proof}
Let $x\in X$ be a point whose distance $\ell$ to the union $Y$ of minimal closed geodesics is as large as possible, and $x'=x\triangleleft\varphi$, $\varphi\in S(A)$ a point in the orbit of $x$. By \Cref{le.dist_between_orbits,le.min_geo_inv} the distance from $x'$ to $Y$ is again $\ell$. Hence, it suffices to prove that the set
\begin{equation*}
\{z\in X\ |\ d(z,Y)=\max_{z'\in X}d(z',Y)\}  
\end{equation*}
is finite; this is what the next lemma does.   
\end{proof}

\begin{lemma}\label{le.max_dist_fin}
In the setting of \Cref{th.main}, let $Y\subseteq X$ be a closed finite union of connected totally geodesic submanifolds. The set of points $x\in X$ where the maximal distance from $Y$ is achieved is finite.  
\end{lemma}

\begin{remark}
We need to allow unions of submanifolds rather than just plain submanifolds in order to be able to apply \Cref{le.max_dist_fin} in the proof of \Cref{le.fin_nonempty}. This is because in principle, two different closed geodesics might intersect.  
\end{remark}

\begin{proof}
We may as well assume $Y\subset X$ is proper, i.e. $\ell>0$ .Write $Y$ as a union of finitely many connected totally geodesic submanifolds $Y_i$ of $X$. 

$X$ can be realized as the quotient $\pi_1\backslash \wt{X}$ of the universal cover of $X$ by the fundamental group $\pi_1=\pi_1(X)$. The preimage of $Y_i$ through the quotient map $p:\wt{X}\to X$ consists of copies $\wt{Y}_i^j$ of the universal cover of $Y_i$ that are translates of one another by the action of $\pi_1$. All $\wt{Y}_i^j$ are closed, connected totally geodesic submanifolds of $\wt{X}$.   

Let
\begin{equation*}
\ell = \max_{x\in X}d(x,Y). 
\end{equation*}
Moving up to the universal cover, the maximality of $\ell$ means that the closed ``tubes'' 
\begin{equation*}
T_i^j = B\left(\wt{Y}_i^j,\ell\right) = \{\wt{x}\in\wt{X}\ |\ d\left(\wt{x},\wt{Y}_i^j\right)\le \ell\}  
\end{equation*}
cover $\wt{X}$ sharply, in the sense that for no $\varepsilon>0$ is the union
\begin{equation*}
\bigcup_{i,j} B\left(\wt{Y}_i^j,\ell-\varepsilon\right)  
\end{equation*}
all of $\wt{X}$. Moreover, the set $S(Y,\ell)\subset X$ of points whose distance from $Y$ is precisely $\ell$ is the image through $p:\wt{X}\to X$ of the set of points $S\left(p^{-1}(Y),\ell\right)$ that are not in the interior of any of the tubes $T_i^j$. We will refer to such points as \define{liminal}.

We will show that the set $S(Y,\ell)$ is discrete; since it is also closed, it must be finite. To this end, fix $x\in S(Y,\ell)$ and let $\wt{x}\in p^{-1}(x)\subset \wt{X}$ be a lift of $x$.  

The set $\cT$ of tubes $T_i^j$ intersecting (and hence covering) a small ball $B$ around $\wt{x}$ is finite; we henceforth restrict our attention to these. 

Suppose there are liminal points $\wt{y}\in B$ arbitrarily close to but different from $\wt{x}$. Such $\wt{y}$ would have to be on the boundary or outside every $T_i^j\in \cT$. Now note that the function $d_i^j:\wt{X}\to \bR_{\ge 0}$ defined by
\begin{equation*}
\wt{X}\ni \wt{y}\mapsto d\left(\wt{y},\wt{Y}_i^j\right)
\end{equation*}
is \define{strictly convex}, in the sense that its composition with every geodesic $\gamma:\bR\to \wt{X}$ satisfies
\begin{equation*}
(d_i^j\circ\gamma)(tu+(1-t)v) < t(d_i^j\circ\gamma)(u) + (1-t)(d_i^j\circ\gamma)(v),\ \forall u,v\in \bR \text{ and } t\in[0,1]   
\end{equation*}
(`strictly' because we have strict inequality). Indeed, \cite[Corollary I.5.6]{Bal95} implies that if $X$ has non-positive curvature then the distance from a convex subset of $\wt{X}$ in the sense of \Cref{def.defns} (such as $\wt{Y}_i^j$) is convex. That proof is easily tweaked to show \define{strict} convexity when the curvature of $X$ is strictly negative.

Now, since $d_i^j\left(\wt{y}\right)\ge \ell = d_i^j\left(\wt{x}\right)$ for all $T_i^j\in\cT$, the strict convexity implies that a point $\wt{z}\in B$ on the geodesic $\gamma$ through $\wt{x}$ and $\wt{y}$ and sitting on the other side of $\wt{y}$ as compared to $\wt{x}$ 
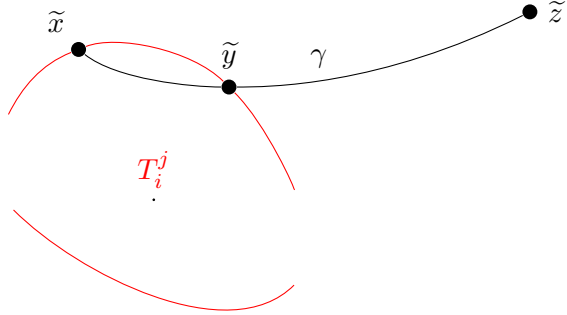
\begin{equation*}
  \begin{tikzpicture}[auto,baseline=(current  bounding  box.center)]
    \draw (0,0) node[circle, inner sep=2pt, fill=black, label={120:{$\wt{x}$}}] (x) {};
    \draw (2,-.5) node[circle, inner sep=2pt, fill=black, label={above:{$\wt{y}$}}] (y) {};
    \draw (6,.5) node[circle, inner sep=2pt, fill=black, label={right:{$\wt{z}$}}] (z) {};
    \draw (1,-2) node[circle, inner sep=0pt, fill=black, label={[red]above:{$T_i^j$}}] () {};
    \node (aux1) at (-1,-1) {};
    \node (aux2) at (3,-2) {};
    \node (aux3) at (-1,-2) {};
    \node (aux4) at (3,-3) {};
    \draw (x)  .. controls +(.5,-.4) and +(-.5,0) .. (y);
    \draw (y)  .. controls +(.5,0) and +(-2,-1) .. node[pos=.5,auto] {$\gamma$}  (z);
    \draw[red] (aux1) .. controls +(.25,.5) and +(-.5,-.2)  .. (x);
    \draw[red] (x) .. controls +(.5,.2) and +(-.5,.5)  .. (s');
    \draw[red] (y) .. controls +(.5,-.5) and +(-.1,.1)  .. (aux2);
    \draw[red] (aux3) .. controls +(1,-1) and +(-1,-1)  .. (aux4);
  \end{tikzpicture}
\end{equation*}
will satisfy $d_i^j\left(\wt{z}\right)>\ell$ for all $T_i^j\in \cT$. But this means that $\wt{z}$ is not contained in any member of $\cT$, contradicting the fact that these tubes cover $B$. 
\end{proof}

\begin{proof_of_main}
We have to show that the set $X_{\cat{fin}}\subseteq X$ consisting of points with finite orbit actually coincides with $X$; this will be a simple matter of putting together the various pieces of the argument we have sketched so far. 

We know from \Cref{le.fin_nonempty,pr.tot_geo} that $X_{\cat{fin}}$ is a non-empty closed totally geodesic submanifold of $X$. But then, if it were proper, \Cref{le.max_dist_fin} would ensure the existence of a point $x\in X\setminus X_{\cat{fin}}$ having finite orbit and hence lead to a contradiction. The conclusion $X_{\cat{fin}}=X$ follows. 
\end{proof_of_main}



\begin{thebibliography}{10}

\bibitem{Bal95}
Werner Ballmann.
\newblock {\em Lectures on spaces of nonpositive curvature}, volume~25 of {\em
  DMV Seminar}.
\newblock Birkh\"auser Verlag, Basel, 1995.
\newblock With an appendix by Misha Brin.

\bibitem{BGS}
Werner Ballmann, Mikhael Gromov, and Viktor Schroeder.
\newblock {\em Manifolds of nonpositive curvature}, volume~61 of {\em Progress
  in Mathematics}.
\newblock Birkh\"auser Boston, Inc., Boston, MA, 1985.

\bibitem{Ban05}
Teodor Banica.
\newblock Quantum automorphism groups of small metric spaces.
\newblock {\em Pacific J. Math.}, 219(1):27--51, 2005.

\bibitem{BhoGos09}
Jyotishman Bhowmick and Debashish Goswami.
\newblock Quantum group of orientation-preserving {R}iemannian isometries.
\newblock {\em J. Funct. Anal.}, 257(8):2530--2572, 2009.

\bibitem{Bic03}
Julien Bichon.
\newblock Quantum automorphism groups of finite graphs.
\newblock {\em Proc. Amer. Math. Soc.}, 131(3):665--673 (electronic), 2003.

\bibitem{Boc95}
Florin~P. Boca.
\newblock Ergodic actions of compact matrix pseudogroups on {$C^*$}-algebras.
\newblock {\em Ast\'erisque}, (232):93--109, 1995.
\newblock Recent advances in operator algebras (Orl{\'e}ans, 1992).

\bibitem{Chi15}
A.~{Chirvasitu}.
\newblock On quantum symmetries of compact metric spaces.
\newblock {\em Journal of Geometry and Physics}, 2015.

\bibitem{Con94}
Alain Connes.
\newblock {\em Noncommutative geometry}.
\newblock Academic Press, Inc., San Diego, CA, 1994.

\bibitem{DasGos13}
B.~{Das}, D.~{Goswami}, and S.~{Joardar}.
\newblock {Rigidity of action of compact quantum groups on compact, connected
  manifolds}.
\newblock {\em ArXiv e-prints}, September 2013.

\bibitem{doC92}
Manfredo~Perdig{\~a}o do~Carmo.
\newblock {\em Riemannian geometry}.
\newblock Mathematics: Theory \& Applications. Birkh\"auser Boston, Inc.,
  Boston, MA, 1992.
\newblock Translated from the second Portuguese edition by Francis Flaherty.

\bibitem{EW14}
Pavel Etingof and Chelsea Walton.
\newblock Semisimple {H}opf actions on commutative domains.
\newblock {\em Adv. Math.}, 251:47--61, 2014.

\bibitem{GHL}
Sylvestre Gallot, Dominique Hulin, and Jacques Lafontaine.
\newblock {\em Riemannian geometry}.
\newblock Universitext. Springer-Verlag, Berlin, third edition, 2004.

\bibitem{Gos12}
D.~{Goswami}.
\newblock {Existence of quantum isometry group for a class of compact metric
  spaces}.
\newblock {\em ArXiv e-prints}, May 2012.

\bibitem{Hua12}
H.~{Huang}.
\newblock {Invariant subsets under compact quantum group actions}.
\newblock {\em ArXiv e-prints}, October 2012.

\bibitem{KanRub57}
L.~V. Kantorovi{\v{c}} and G.~{\v{S}}. Rubin{\v{s}}te{\u\i}n.
\newblock On a functional space and certain extremum problems.
\newblock {\em Dokl. Akad. Nauk SSSR (N.S.)}, 115:1058--1061, 1957.

\bibitem{Kan42}
L.~Kantorovitch.
\newblock On the translocation of masses.
\newblock {\em C. R. (Doklady) Acad. Sci. URSS (N.S.)}, 37:199--201, 1942.

\bibitem{KusTus99}
Johan Kustermans and Lars Tuset.
\newblock A survey of {$C^*$}-algebraic quantum groups. {I}.
\newblock {\em Irish Math. Soc. Bull.}, (43):8--63, 1999.

\bibitem{PauTom09}
Vern~I. Paulsen and Mark Tomforde.
\newblock Vector spaces with an order unit.
\newblock {\em Indiana Univ. Math. J.}, 58(3):1319--1359, 2009.

\bibitem{QuaSab12}
Johan Quaegebeur and Marie Sabbe.
\newblock Isometric coactions of compact quantum groups on compact quantum
  metric spaces.
\newblock {\em Proc. Indian Acad. Sci. Math. Sci.}, 122(3):351--373, 2012.

\bibitem{Tak02}
M.~Takesaki.
\newblock {\em Theory of operator algebras. {I}}, volume 124 of {\em
  Encyclopaedia of Mathematical Sciences}.
\newblock Springer-Verlag, Berlin, 2002.
\newblock Reprint of the first (1979) edition, Operator Algebras and
  Non-commutative Geometry, 5.

\bibitem{Dae95}
A.~Van~Daele.
\newblock The {H}aar measure on a compact quantum group.
\newblock {\em Proc. Amer. Math. Soc.}, 123(10):3125--3128, 1995.

\bibitem{DaeWan96}
Alfons Van~Daele and Shuzhou Wang.
\newblock Universal quantum groups.
\newblock {\em Internat. J. Math.}, 7(2):255--263, 1996.

\bibitem{Wan98}
Shuzhou Wang.
\newblock Quantum symmetry groups of finite spaces.
\newblock {\em Comm. Math. Phys.}, 195(1):195--211, 1998.

\bibitem{Wan99}
Shuzhou Wang.
\newblock Ergodic actions of universal quantum groups on operator algebras.
\newblock {\em Comm. Math. Phys.}, 203(2):481--498, 1999.

\bibitem{Wor87}
S.~L. Woronowicz.
\newblock Compact matrix pseudogroups.
\newblock {\em Comm. Math. Phys.}, 111(4):613--665, 1987.

\end{thebibliography}
\bibliographystyle{plain}
\addcontentsline{toc}{section}{References}

\def\polhk#1{\setbox0=\hbox{#1}{\ooalign{\hidewidth
  \lower1.5ex\hbox{`}\hidewidth\crcr\unhbox0}}}

\end{document}